\documentclass[11pt]{article}
\usepackage{amsmath,amsthm,amsfonts,amssymb,mathrsfs,bm}
\usepackage{amsmath,amsthm,amsfonts,amssymb,bm,wasysym}
\usepackage{epsfig}
\usepackage[usenames]{color}
\usepackage{verbatim}
\usepackage{hyperref}
\usepackage{multicol}
\usepackage{graphicx}
\usepackage{comment}
\usepackage{float}
\usepackage{color}
\usepackage{enumerate}
\usepackage[normalem]{ulem}

\usepackage[color=green, textsize=tiny]{todonotes}

\parskip \medskipamount
\newtheorem{theorem}{Theorem}[section]
\newtheorem{definition}[theorem]{Definition}

\newtheorem{lemma}[theorem]{Lemma}
\newtheorem{proposition}[theorem]{Proposition}
\newtheorem{corollary}[theorem]{Corollary}
\newtheorem{remark}[theorem]{Remark}

\newtheorem{claim}[theorem]{Claim}

\numberwithin{equation}{section}

\def\N{\mathbb{N}}
\def\Z{\mathbb{Z}}

\def\bP{\mathbb{P}}
\def\bE{\mathbb{E}}
\def\F{\mathcal{F}}

\newcommand{\cpc}[1]{\mathrm{BCap}(#1)}

\def\cU{\mathcal{U}}
\def\cT{\mathcal{T}}
\def\cS{\mathcal{S}}

\def\cF{\mathcal{F}}

\def\cC{\mathcal{C}}

\def\cTp{\mathcal{T}_-}

\def\cTc{\mathcal{T}_c}

\def\reff#1{(\ref{#1})}

\renewcommand{\phi}{\varphi}
\renewcommand{\epsilon}{\varepsilon}

\allowdisplaybreaks

\newcommand{\1}{{\text{\Large $\mathfrak 1$}}}

\newcommand{\til}{\widetilde}

\newcommand{\pr}[1]{\mathbb{P}\!\left(#1\right)}
\newcommand{\E}[1]{\mathbb{E}\!\left[#1\right]}
\newcommand{\estart}[2]{\mathbb{E}_{#2}\!\left[#1\right]}
\newcommand{\prstart}[2]{\mathbb{P}_{#2}\!\left(#1\right)}
\newcommand{\prcond}[3]{\mathbb{P}_{#3}\!\left(#1\;\middle\vert\;#2\right)}

\def\bs{\backslash}
\newcommand{\norm}[1]{\left\| #1 \right\|}

\newcommand{\red}[1]{{\color{red}{#1}}}

\newcommand{\diam}[1]{{\rm{diam}}(#1)}

\newcommand{\musb}{\mu_{\mathrm{sb}}}
\newcommand{\offp}{k_{\mathrm{p}}}
\newcommand{\offf}{k_{\mathrm{f}}}

\begin{document}

\title{\bf Local times and capacity for transient branching random walks}

\author{Amine Asselah \thanks{
Universit\'e Paris-Est, LAMA, UMR 8050, UPEC, UPEMLV, CNRS, F-94010 Cr\'eteil; amine.asselah@u-pec.fr} \and
Bruno Schapira\thanks{Aix-Marseille Universit\'e, CNRS, I2M, UMR 7373, 13453 Marseille, France;  bruno.schapira@univ-amu.fr} \and Perla Sousi\thanks{University of Cambridge, Cambridge, UK;   p.sousi@statslab.cam.ac.uk} 
}
\date{}
\maketitle

\begin{abstract}  
We consider branching random walks on the Euclidean lattice in dimensions five and higher.  
In this {\it non-Markovian} setting, we first obtain a relationship between the equilibrium 
measure and Green's function, in the form of an approximate last passage decomposition. 
Secondly, we obtain exponential moment bounds for functionals of the branching random walk, under optimal condition. 
As a corollary we obtain an approximate variational characterisation of the branching capacity. We finally derive upper bounds involving the branching capacity for the tail of the time spent in an arbitrary finite collection of balls. This generalises the results of~\cite{AHJ} and~\cite{AS22} for~$d\geq 5$. For random walks, the analogous tail estimates have been instrumental 
tools for tackling deviations problems on the range, related to folding of the walk.   
\newline
MSC 2010 \emph{subject classifications.} Primary  60G50; 60J80.
\end{abstract}

\red{}

\section{Introduction}

In this paper we study branching random walks (BRW), also called tree-indexed random walks, on $\Z^d$ with $d\geq 5$. To define this process, we need two sources of randomness. First we sample a random spatial rooted tree and next we attach i.i.d.\ random variables to the edges of the tree. The branching random walk is then obtained by assigning to each node of the tree the sum of all the variables associated to the edges along the unique geodesic path from the root to that node. 

To be more precise, for a general ordered rooted tree, we denote the root by $\emptyset$ and the parent of a vertex $u\neq \emptyset$ by $u^-$.
Given $x\in \mathbb Z^d$, a  
random walk indexed by a (possibly random) rooted tree~$T$ starting from $x$ is a set of random variables $\{S^x_u\}_{u\in T}$ indexed by the vertices of $T$ with values in $\mathbb Z^d$, which is 
such that, given $T$, $S_\emptyset^x= x$, and the set of 
 increments $\{S_u^x-S_{u^-}^x\}_{u\in T\setminus\{\emptyset\}}$, forms a family of independent and identically distributed random variables. When $x=0$, we sometimes drop it from the notation. 

To keep our analysis simpler, we assume in the whole paper that the joint distribution of the increments of our tree-indexed walks is given by the uniform measure on nearest neighbours of the origin.
We note that all our proofs and results would adapt to centred finitely supported distributions.

We denote by $T^x$ the range of the random walk indexed by $T$ starting from $x$, i.e.\  
$$T^x = \{S_u^x : u\in T\}.$$

In our work we consider two types of random rooted spatial trees: a critical Bienaym\'e-Galton-Watson tree and an \emph{invariant} infinite tree that we now define. 

Let $\mu$ be an offspring distribution with mean $1$ and positive finite variance $\sigma^2$. We write $\musb$ for the size biased distribution of $\mu$, i.e.\ $\musb(i)=i\mu(i)$ for $i\in \N$. 
\begin{definition}
\rm{
Let $\cT$ be an infinite ordered and rooted tree constructed as follows: 
\begin{itemize}
	\item The root produces $i$ offspring with probability $\mu(i-1)$ for every $i\ge 1$. The first offspring of the root is \emph{special}, while the others if they exist are \emph{normal}. 
	\item Special vertices produce offspring independently according to $\musb$, while normal vertices produce offspring independently according to $\mu$. 
	\item Each special vertex produces exactly one special vertex chosen uniformly at random among its children, while the other children are normal.
\end{itemize}
Note that if we forget the spatial structure of $\cT$, we obtain a prominent example of an invariant one-ended tree introduced by Aldous in~\cite{Ald91}, which appears as the local limit as $n\to \infty$ of a 
Galton-Watson tree with offspring distribution $\mu$
conditioned on having $n$ vertices, and rooted at a uniformly chosen vertex, see again~\cite{Ald91}.

By construction $\cT$ has a unique infinite path stemming from the root that we call the \emph{spine}. We assign label $0$ to the root. We assign positive labels to the vertices on the side of the spine reached clockwise from the root according to a depth-first search from the root and negative labels to the other ones according to depth-first search from infinity as depicted in the first tree of Figure~\ref{fig:labelling}.
 We call the vertices with negative labels (including the spine vertices) the \emph{past} of $\cT$ and denote them by~$\cT_-$, while the vertices with non-negative labels are in the 
\emph{future} of $\cT$ and we denote them by~$\cT_+$. 
Note that the root does not have any offspring in the past of~$\cT$.
}	
\end{definition}

\begin{figure}[ht!]\label{fig:labelling}
	\begin{center}
		\includegraphics[width=0.31\textwidth]{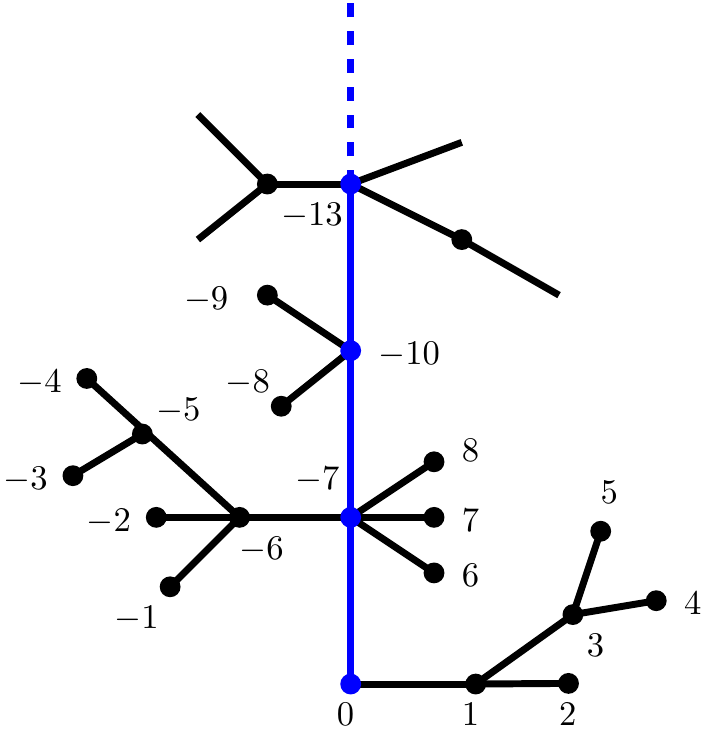}\hspace{0.02\textwidth}
		\includegraphics[width=0.31\textwidth]{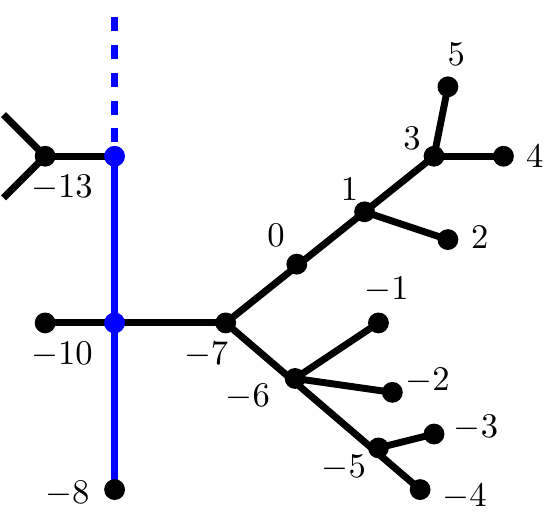}\hspace{0.02\textwidth}
	\includegraphics[width=0.31\textwidth]{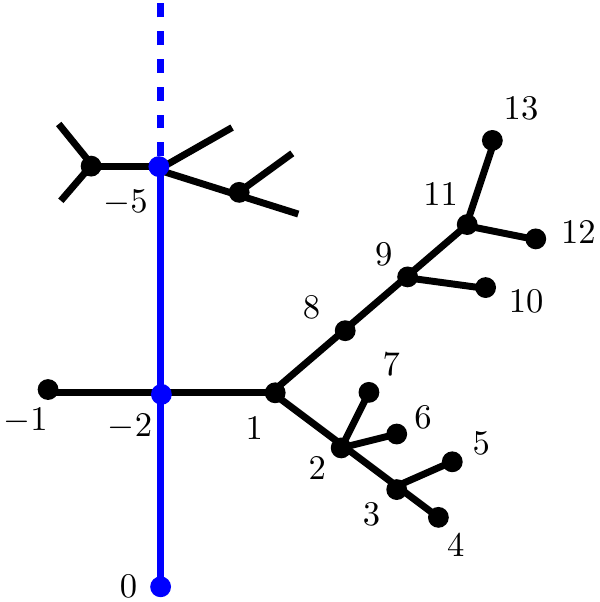}
	\end{center}
	\caption{An infinite tree rooted at $0$, seen from $-8$ and relabelled.}
\end{figure}

We denote by $\cT_c$ a Bienaym\'e-Galton Watson tree with offspring distribution $\mu$. 
The tree $\cT_c$ 
is almost surely finite, but conditioned on being large, say having
$n$ nodes, it has of the order of~$\sqrt n$ generation and the set of positions, $\cT^x_c$, typically fills a ball
of volume~$n^{d/4}$. The cases $d\ge 5$ are called by physicists the {\it upper-critical space dimensions}, when
considering the BRW. For probabilists, when~$d\ge 5$ the critical BRW is {\it transient} in the sense that a BRW, conditioned on having size~$n$, visits the origin a finite number of times independent of $n$, and the expected number of visits of a site (the so-called Green's function) is well defined. Dimension four is the critical dimension for BRW, and there the number of visits to the origin grows logarithmically (in the size of the critical tree). We mention here the works of Le Gall and Lin \cite{LGL1,LGL2} in the transient dimensions that obtained  laws of
large numbers for the volume of the set of visited sites, when conditioning the BRW to have $n$ nodes,
as $n$ goes to infinity. The far-reaching idea behind the law of large numbers is introducing an infinite
labelled tree invariant under a shift of the labelling as shown in Figure~\ref{fig:labelling}, and its corresponding re-rooting of the tree. On such an invariant
object, such as $\cT_-$, ergodic theory, yields laws of large numbers. The genealogy of
the invariant tree looks like a comb whose teeth are independent critical trees, with respective
volumes being independent heavy tailed variables 
(since by Kolmogorov's estimate $\bP(|\cT_c|> t)\asymp \frac{1}{\sqrt t}$). This allowed
Le Gall and Lin to retrieve information on the critical tree from the infinite invariant ones. Then, building on
their beautiful observations, Zhu \cite{Zhu1} defines, in $d\ge 5$, the branching capacity of a set, and links it to
the probability that a critical BRW (indexed by~$\cT_c$) or an infinite BRW (indexed by $\cT_-$) hits the set, properly normalised. With our notation, 
Zhu's key results read as follows.
\begin{equation}\label{intro-1}
\lim_{\|x\|\to\infty} \frac{\bP(\cT_c^x\text{ hits } K)}{g(x)}\, =\, \cpc{K}\, =\,
\lim_{\|x\|\to\infty} \frac{\bP(\cT^x_-\text{ hits } K)}{G(x)},
\end{equation}
where $g$ (respectively $G$) is the Green's function for the critical tree $\cT_c^x$ (resp.\ for the infinite tree~$\cT_-^x$).
In other words,
\[
g(x)=\E{\sum_{u\in \cT_c} \1(S_u^0=x)},\quad \text{and}\quad 
G(x)= \E{\sum_{u\in \cT_-} \1(S_u^0=x)}.
\]
Note that by criticality of the tree, the function $g$ is the same as the Green's function for simple random walk.

Moreover, \cite{Zhu1} gives a dual definition of branching capacity in terms of escape probabilities:
\begin{equation}\label{intro-2}
\cpc{K}=\sum_{x\in K} \bP(\cT_-^x\cap  K=\emptyset).
\end{equation}
In view of these findings, a transient branching random walk seems to resemble a transient random walk. What about a last passage decomposition? Let us recall what it amounts to for a simple
random walk, denoted here by $(X_n)_{n\ge 0}$. When starting at $x$, we denote its law by $\bP_x$, and 
for any subset $K\subset \Z^d$, we denote by $H_K=\inf\{n\ge 0 : X_n\in K\}$ 
and $H_K^+ = \inf\{n\ge 1: X_n\in K\}$, the first hitting and the first return time to $K$ respectively.
Then we have a last passage decomposition 
\begin{equation}\label{intro-3}
\bP_x(H_K<\infty) =\sum_{y\in K} \sum_{n\in \N}\bP_x\big(X_n=y,\ \{X_{n+k}: k\ge 1\} \text{ avoids } K\big).
\end{equation}
We next use the Markov property and invariance by time shift to get, for any $x\in \Z^d$
\begin{equation}\label{intro-4}
\bP_x(H_K<\infty)=\sum_{y\in K} g(x-y)\cdot \bP_y(H_K^+=\infty).
\end{equation}
For a tree this Markov property fails and this is the main source of difficulty in this setting.
The fundamental formula \reff{intro-4} is at the heart of the potential theory for random walks,
and can also be thought of as a relation between the function $x\mapsto \mathbb P_x(H_K<\infty)$, which equals
one on $K$ and is harmonic outside $K$, the Green's function $g$ and the equilibrium measure.
Such an exact formula in the context of branching random walks is missing due to the lack of Markov property, 
and so far it has prevented the development 
of a satisfactory potential theory for branching random walks, despite the series of works by Zhu~\cite{Zhu1,Zhu2,Zhu18,Zhu4, Zhu20,Zhu3} that laid its foundations.  
The main goal of this paper is to establish a relation between the Green's function and the equilibrium 
measure similar to \reff{intro-4}. As a main application we are able to establish an \textit{approximate variational characterisation}, see Corollary~\ref{thm:variationalchar}. 

Our second result is obtained independently of the approximate last passage decomposition and provides exponential moment bounds on certain functionals of the BRW for both the infinite tree and the critical one. 
Using these results, we estimate the probability the BRW spends a large time
in each of the balls making up a domain $\Lambda$, in terms of the branching capacity 
of $\Lambda$, in an analogous form as for the simple random walk. 

We then give two other corollaries of our
results:  (i) Euclidean balls are sets of minimal branching capacity (up to constant), given their volume, and 
(ii) in each finite set $\Lambda\subset \Z^d$, there is a subset whose branching capacity and volume are of
order the branching capacity of $\Lambda$, which were instrumental for solving large deviations problems on the range, see below. 

Before stating precisely our results, we recall the definitions of equilibrium measure and branching capacity, which were introduced by Zhu in~\cite{Zhu1}. We assume from now on and until the end of the paper that $d\ge 5$. 
\begin{definition}
	\rm{
	Let $K$ be a finite subset of $\mathbb Z^d$. The {\bf equilibrium measure} of $K$ is the measure defined for $x\in \mathbb Z^d$, by 
	\[ 
	e_K(x) = \1(x\in K)\cdot  \mathbb P(\cT_-^x\cap K=\emptyset).
	\]
	The {\bf branching capacity} of $K$ is defined to be 
	\[
	\cpc{K}=\sum_{x\in K}e_K(x).
	\]
	}
\end{definition}
Our first result provides an {\it approximate last passage decomposition}. 

\begin{theorem} \label{theo.potentiel}
\rm{	There exist positive constants $c$ and $C$, so that for any finite set $K\subseteq \Z^d$, we have
	\begin{equation}\label{Upper.theo}
	\|G e_K\|_\infty = \max_{x\in \mathbb Z^d}\, Ge_K(x):= \max_{x\in \mathbb Z^d} \sum_{y\in K} G(x,y) e_K(y) \leq C, 
	\end{equation}
	and if $K$ is nonempty, 
	\begin{equation}\label{Lower.theo}
	\min_{x\in K}\, Ge_K(x) \ge c. 
	\end{equation}
	}
\end{theorem}
One application of this result is an approximate variational characterisation of the branching capacity, 
which solves an open question of~\cite{Zhu20}. 
\begin{corollary}\label{thm:variationalchar}
\rm{	There exist positive constants $C_1,C_2$, such that for 
any finite nonempty set $K\subseteq \Z^d$, 
\begin{equation}\label{capacity-energy}
	\frac{C_1}{\cpc{K}} \le  \inf\Big\{ \sum_{x,y\in K} G(x-y) \, \nu(x)\nu(y) : \nu \textrm{ probability measure on }K\Big\} \le  \frac{C_2}{ \cpc{K}}.
	\end{equation}
	}
\end{corollary} 
We shall give two proofs of this corollary. One, based on the method of~\cite{BPP95}, which only uses~\eqref{Upper.theo}, but requires a finite third moment on $\mu$, and a second one which needs both~\eqref{Upper.theo} and~\eqref{Lower.theo}, but only requires a finite second moment on $\mu$.


Our second main result is an exponential moment bound for functionals of the infinite BRW, which can be thought of as an extension of Kac's moment formula 
to a non-Markovian setting, see e.g.~\cite[Proposition 2.9]{S12}, and is obtained independently of Theorem~\ref{theo.potentiel}. 

\begin{theorem}\label{theo.expmoment}
\rm{Assume that $\mu$ has a finite exponential moment. There exists $\kappa>0$, such that for any map $\varphi: \mathbb Z^d \to [0,\infty)$, satisfying 
$\|G\varphi\|_\infty \le \kappa$, one has 
$$\E{ \exp\Big(\sum_{u\in \cT} \varphi(S_u) \Big) } \le 2. $$ 
}
\end{theorem} 
Theorem~\ref{theo.expmoment} is useful if we can find a {\it potential} $\varphi$ with 
$\|G\varphi\|_\infty \le \kappa$, and Theorem~\ref{theo.potentiel} provides such a family of functions:
we build them from the equilibrium potential $e_K$, with appropriate $K$.

Our proof of Theorem~\ref{theo.expmoment} uses a moment method, which bears similarities with the one in~\cite{AHJ}. However, while \cite{AHJ} handles any dimension, including the intricate dimension four, it restricts its analysis to the case when $\varphi$ is supported on a single point. Our approach is a priori restricted to dimensions five and higher but we are able to consider any function satisfying  $\|G\varphi\|_\infty \le \kappa$, which is ideally suited to study covering of a given domain of space by BRW, and our recursion method is elementary.

When combined together our two main results have a number of interesting consequences. A first immediate application is a large deviations upper bound for the time spent on a collection of balls. For $x\in \mathbb Z^d$, and $r>0$, write $B(x,r)$ for the Euclidean closed ball of radius $r$ centered at $x$ (intersected with $\mathbb Z^d$), and for a subset $\cC\subset \mathbb Z^d$, define $B(\cC,r)=\cup_{x\in \cC} B(x,r)$.  

For $x,y\in \mathbb Z^d$, we define 
$$
\ell_{\cT^x}(y) = \sum_{u\in \cT} \1(S^x_u = y),$$ 
and for $A\subseteq \mathbb Z^d$, we write $\ell_{\cT^x}(A) = \sum_{y\in A} \ell_{\cT^x}(y)$. 

 \begin{corollary}\label{cor.loctimeball}
\rm{Assume that $\mu$ has a finite exponential moment. Then there exists $\kappa>0$, such that for any $t>0$, any $r\ge 1$, and any finite set $\cC\subset \mathbb Z^d$, 
\begin{equation}\label{ineq-cor-localtime}
\mathbb P\Big(\ell_{\cT^0}(B(x,r)) > t, \text{ for all }x\in \cC\Big) \le 2 \exp\Big( -\kappa \cdot \frac{t}{r^d} \cdot \cpc{B(\cC,r)} \Big).
\end{equation}
 } 
\end{corollary}
\begin{remark}\label{rem-localtimes}
\emph{Note that since $\cT$ contains $\cT_c$ as subtree, the same result holds as well with $\cT_c^0$ instead of $\cT^0$ (and this remark applies to Theorem~\ref{theo.expmoment} as well).  }
\end{remark}
Also, it was proved in~\cite{Zhu1} that the branching capacity of a 
ball of radius $r$ is of order $r^{d-4}$, thus in the case of one ball we recover the results of~\cite{AHJ, AS22}.
We note that the case of dimension four would require a completely different strategy, and is open (except for one ball). The question of obtaining a corresponding lower bound is also interesting, but even in the case of a single random walk, which form the lower bound should take is not clear to us.
 
Let us mention that in the setting of a simple random walk the argument used here for proving Corollary~\ref{cor.loctimeball} provides an alternative and more direct approach for Theorem 1.2 from~\cite{AS23a}, see also Remark~\ref{rem.balls.rw}. 

We present now a result which has proved useful in studying the folding of a random walk in large deviations problems on the range, as in~\cite{AS23a}, 
and also played an important role in the context of random interlacements~\cite{S21,S23}. 
\begin{corollary}\label{cor-subset}
\rm{There exists $\alpha>0$, such that for any $r\ge 1$ and any finite $\mathcal C\subset \Z^d$, there is
a subset~$U\subseteq \mathcal C$, satisfying 
\begin{equation}\label{ineq-cor}
(i)\quad\cpc{\cup_{x\in U} B(x,r)}\ge \alpha\cdot  r^{d-4} |U|\quad\text{and}\quad
(ii)\quad r^{d-4} |U|\ge\alpha\cdot \cpc{\cup_{x\in \mathcal C} B(x,r)}.
\end{equation}
}
\end{corollary}
Another application of Theorem~\ref{theo.expmoment} is the following 
general upper deviations bound for the time spent in an {\it arbitrary} set. Define for $x\in \mathbb Z^d$, and $\Lambda\subset \mathbb Z^d$, $G(x,\Lambda) = \sum_{y\in \Lambda} G(x,y)$. 
\begin{corollary}\label{cor.locset}
\rm{Assume that $\mu$ has a finite exponential moment. Then there exists $c>0$, such that for any $t>0$ and any finite nonempty set $\Lambda\subset \mathbb Z^d$, one has 
$$\mathbb P(\ell_{\cT^0}(\Lambda) > t)  \le 2 \exp\left(-c \cdot\frac{t}{\sup_{x\in \Z^d} G(x,\Lambda)}\right). $$ 
}
\end{corollary} 
Note that $G(x,\Lambda)\le C|\Lambda|^{4/d}$, for some universal constant $C$ which does not depend on $\Lambda$ nor on $x\in \mathbb Z^d$, thus in the above corollary, one could as well replace the denominator in the exponential by $|\Lambda|^{4/d}$. However, the slightly more general form presented here can be useful in some situations when e.g. the set $\Lambda$ has a small local density, as in~\cite{AS21,AS23b}.

In the standard framework of random walks, Corollaries~\ref{cor.loctimeball}, \ref{cor-subset} and~\ref{cor.locset} were instrumental in studying the moderate deviations of the range in \cite{AS17,AS21,AS23a}, and for solving a long-standing conjecture of Khanin, Mazel, Schlossman and Sinai~\cite{KMSS94} concerning the large deviations for the intersection of two independent ranges in~\cite{AS23b}. We expect that our tools will permit analogous questions 
to be handled in upper-critical dimensions for BRW.


For completeness and to emphasise the analogy with a single random walk we provide two other variational characterisations. 
\begin{corollary}\label{cor:variation}
\rm{There exist positive constants $C_3,C_4$, such that for any finite set $K\subset \mathbb Z^d$, 
\begin{equation}\label{caract.2}
C_3\cdot \cpc{K} \le  \sup  \Big\{\sum_{x\in K} \varphi(x) : \varphi\ge 0 \text{ on }K, \, \max_{x\in K} G\varphi(x) \le 1 \Big\} \le C_4 \cdot\cpc{K}, 
\end{equation} 
and 
\begin{equation}\label{caract.3}
C_3\cdot \cpc{K} \le  \inf  \Big\{\sum_{x\in K} \varphi(x) : \varphi \ge 0 \text{ on }K,  \, \min_{x\in K} G\varphi (x) \ge 1\Big\} \le C_4\cdot \cpc{K}.
\end{equation}
}
\end{corollary}
Finally another immediate application of Theorem~\ref{theo.potentiel} is the following general lower bound for the branching capacity of a set in terms of its volume. 
\begin{corollary}\label{cor.lowerbcap}
\rm{ There exists a positive constant $c>0$, such that for any finite set $K\subset \mathbb Z^d$, 
$$\cpc{K}  \ge c|K|^{1-\frac 4d}. $$ 
}
\end{corollary} 
Together with the fact already mentioned that the branching capacity of a ball $B(x,r)$ is of order~$r^{d-4}$, this shows that as for the usual Newtonian capacity, balls are sets having minimal branching capacity among those with fixed volume (at least up to universal constants). 

\textbf{The non-Markovian nature of the BRW.} We mentioned already that this work can be seen as
extending some key random walk potential theoretic results to a non-Markovian setting: 
the infinite invariant tree defined
above. We describe the latter informally as being a one-sided infinite random walk, 
the so-called spine, on each node of which there are two critical trees hanging off, one in the past and the other one in the future. 
It is only when conditioning on the spine that we can disentangle past and future.
However, as we average over both the spine and its dangling trees there is often a subtle tradeoff between
what is required from the spine, and what the dangling trees achieve.
We use extensively that we are in a regime where the typical behaviour is dominant, and the spine
is a simple random walk which typically spends time $L^2$ in a region of diameter $L$. If we can place ourselves
in a situation where the dangling trees all see the same environment (typically avoid some set far away) then
their initial position is innocuous, and we can use uniform bounds on them. 

\textbf{Organisation.} The rest of the paper is organised as follows. 
In Section~\ref{sec.Prel} we provide some preliminary results, such as the shift-invariance of the infinite tree, some basic facts about Green's functions, and also about the equilibrium measure. We also recall important results of Zhu on hitting probability estimates. In Section~\ref{sec.Proof1} we first give a sketch of the proof of~\eqref{Upper.theo} and in the remaining section we give the full proof as well as provide a first proof of Corollary~\ref{thm:variationalchar}, that only uses the upper bound part of Theorem~\ref{theo.potentiel}. 
In Section~\ref{sec.Moment} we prove Theorem~\ref{theo.expmoment}, together with Corollary~\ref{cor.loctimeball}. Then in Section~\ref{sec.Lowerbound} 
we prove~\eqref{Lower.theo}, which concludes the proof of Theorem~\ref{theo.potentiel}. Finally in Section~\ref{sec.cor}, we give another short proof of Corollary~\ref{thm:variationalchar}, and prove the remaining results, Corollaries~\ref{cor-subset}, \ref{cor.locset}, \ref{cor:variation}, and~\ref{cor.lowerbcap}. 
 

\section{Preliminaries}\label{sec.Prel}
\subsection{General notation}
Given two real functions $f$ and $g$, we write $f\lesssim g$, or sometimes $f=\mathcal O(g)$, when there exists a constant $C>0$, such that $f(x) \le Cg(x)$, for all $x$. We write $f\asymp g$, when both $f\lesssim g$ and $g\lesssim f$. We write $f\sim g$, when $f(x)/g(x) \to 1$, as $x\to \infty$. 

Given $A\subset \mathbb Z^d$, we let $A^c := \mathbb Z^d\bs A$, and we define the boundary $\partial A$ of $A$ as the set of elements of~$A$ which have at least one neighbor in $A^c$.

We let $\|x\|$ denote the Euclidean norm of an element $x\in \mathbb Z^d$. We write $a\wedge b$ and $a\vee b$ respectively for the minimum and maximum between two real numbers $a$ and $b$. For $r>0$, we write $G(r) = r^{4-d}$.

\subsection{Trees and tree-indexed random walks}
Let $\til{\mu}$ be the probability measure defined by $\til{\mu}(i)=\sum_{j\geq i+1} \mu(j)$. 
A tree where only the distribution of the offspring of the root is $\til{\mu}$ and everywhere else it is $\mu$ is called a {\bf $\mu$-adjoint tree} and we denote it by $\til{\cT}_c$.

\begin{definition}
	\rm{
	We define the {\bf shift transformations} $\theta$ and its inverse $\theta^{-1}$ on $\cT$ by adding $1$, respectively $-1$, to all labels and then the vertex with label $0$ becomes the new root. Furthermore, these maps extend naturally to transformations on the law of the random walk indexed by $\cT$.  
	}
\end{definition}

\begin{proposition}\label{pro:shiftinvariance}
	\rm{The laws of $\cT$ and of the random walk indexed by $\cT$ are invariant under the shift transformations $\theta$ and $\theta^{-1}$. }
\end{proposition}
We refer the reader to \cite{Zhu18} or~\cite{BW22} for a full proof of this proposition. The fact that the law of the random walk indexed by $\cT_+$ is invariant under $\theta^{-1}$ was first observed by Le Gall and Lin \cite{LGL1,LGL2}, and is reminiscent of the invariance property of sin-trees discovered by Aldous \cite{Ald91}.

\begin{figure}[ht!]\label{fig:pastfuturegreenred}
	\begin{center}
		\includegraphics[scale=1]{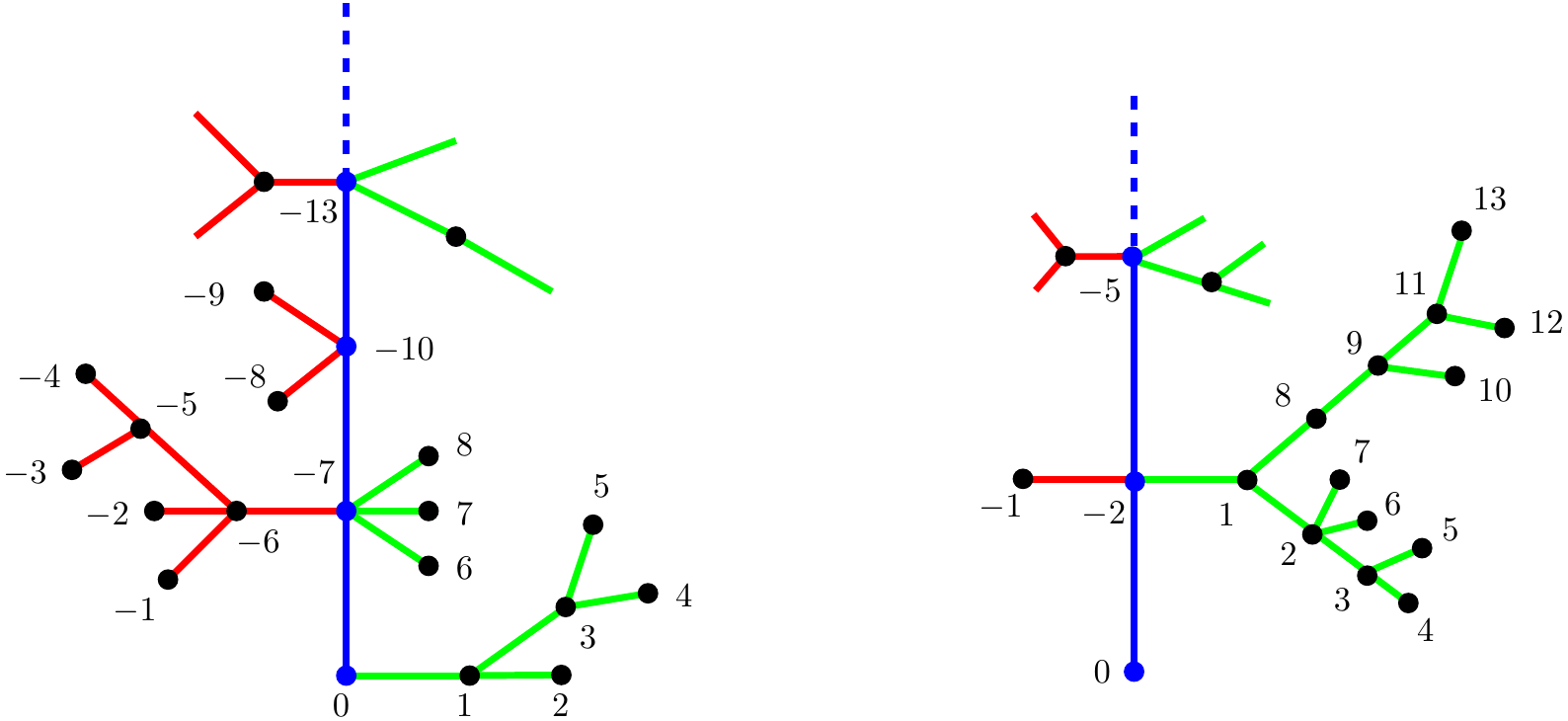}
		\end{center}
		\caption{		Applying the shift transformation $\theta^8$: the blue part is the spine, the green part is the future and the red part is the past.}
		\end{figure}

%
%
%

For a special vertex $u$ of $\cT$ (i.e. on the spine of $\cT$) we write $\offp(u)$ and $\offf(u)$ for the number of normal offsprings of $u$ in the past and future of $\cT$ respectively. By the construction of $\cT$ we then see that for all $i,j\in \N$, 
\[
\pr{\offp(u)=i, \offf(u)=j}=\mu(i+j+1),
\]
and hence both $\offp(u)$ and $\offf(u)$ are distributed according to $\til{\mu}$. As a consequence, the subtrees of~$\cT$ 
emanating from the vertices on the spine, either in the past or in the future, are $\mu$-adjoint trees. A random walk indexed by a $\mu$-adjoint tree is called an {\bf adjoint branching random walk}.

For $n\ge 0$, we let $\cT^x(n)$ be the value of the random walk indexed by $\cT$ starting from $x$ at the vertex with label $n\in \mathbb Z$, 
and similarly for the other tree-indexed walks. 
For $a\le b$ in $\mathbb Z$, we write $\cT^x[a,b]:=\{\cT^x(n)\}_{n\in [a,b]}$, and similarly for $\cT^x[a,b)$, $\cT^x(a,b]$, or for other random trees.  
We write $(X^x(n))_{n\in \N}$ for the random walk indexed by the spine parametrised by its intrinsic labelling (i.e.\ its natural time parametrisation), when it starts from $x\in \mathbb Z^d$.

With a slight abuse of notation we shall sometimes denote the tree-indexed random walk (as a process) by $\cT^x$ (which was formally defined as a random subset of $\mathbb Z^d$). 
For integers $0\le a \le b\le \infty$, we also write 
	$\F_-^x[a,b]$ (respectively $\F_+^x[a,b]$) for the positions of $\cT^x$ on the set of vertices lying in the adjoint trees in the past (respectively future) emanating from the points on the spine with time index (in the natural parametrisation of the spine) in $[a,b]$, and similarly with $[a,b)$. See Figure~\ref{fig:pastfuturegreenred} for an illustration of these definitions.


\subsection{Basic facts on Green's functions and branching capacity}
We denote by $g$ the Green's function of a simple random walk $(X_n)_{n\ge 0}$, in other words: 
$$g(x,y) = \mathbb E_x\Big[\sum_{n=0}^\infty \1(X_n = y)\Big]= g(0,y-x),$$
where $\mathbb E_x$ denotes the expectation for a walk starting from $x$.  
Recall that (see e.g.~\cite{LL10}), 
\begin{equation}\label{green.rw}
g(x,y) \asymp \frac 1{1+\|x-y\|^{d-2}}.
\end{equation} 
In fact by linearity of expectation and criticality of $\mu$, one also has 
$$g(x,y ) =\mathbb E\Big[\sum_{u\in \cT_c} \1(S_u^x = y)\Big]. $$ 
Similarly we define 
$$G(x,y ) = G(y-x) = \mathbb E\Big[\sum_{u\in \mathcal T_-} \1(S_u^x = y)\Big]. $$ 
Since the mean number of normal offspring of a vertex on the left of the spine of $\cT$ has mean $\sigma^2/2$, 
we deduce that for any $x,y\in \mathbb Z^d$, 
\begin{align}\label{convol.G} 
\nonumber G(x,y) & =  \sum_{z\in \mathbb Z^d} \Big(g(x,z) - \1(z=x)\Big) \cdot \Big(\1(z=y) + \frac{\sigma^2}{2} ( g(z,y) - \1(z=y))\Big) \\
& = \frac{\sigma^2}{2} \sum_{z\in \mathbb Z^d} g(x,z)g(z,y) + \mathcal O(g(x,y)).  
\end{align}
Using~\eqref{green.rw},  this yields 
\begin{equation}\label{Green.asymp}
G(x,y) \asymp \frac 1{1+ \|x-y\|^{d-4}}. 
\end{equation}
We now state some important facts related to the equilibrium measure and the branching capacity.
The next proposition is a simple last passage decomposition formula which will be used widely in the paper.

\begin{proposition}\label{pro:lastpassagedec}
	Let $K'\subseteq K$ and $B$ be finite subsets of $\Z^d$ such that $B$ contains $K'$ and points at distance $1$ from $K'$. Then 
	\[
	\sum_{y\in K'} e_{K}(y) = \sum_{w\in \partial B} \pr{\cT_-^w\cap (K\cup B)=\emptyset, \cT^w_+ \text{ first hits } K \text{ in } K'}.
	\]
\end{proposition}

\begin{proof}[\bf Proof]
Let $L$ be the last time that the past of the tree-indexed random walk is in $B$. For all~$y\in K'$ we can now write 
\begin{align*}
	e_K(y) = \pr{\cT^y_-\cap K =\emptyset} = \sum_{w\in \partial B} \sum_{n\geq 1} \pr{\cT_-^y\cap K=\emptyset, \cT^y(-n)=w, L=n},
\end{align*}
since $L$ is finite almost surely.
	Using the invariance of the tree $\cT^y$ under the shift $\theta^n$ by Proposition~\ref{pro:shiftinvariance} we get 
	\[
	\pr{\cT_-^y\cap K=\emptyset, \cT^y(-n)=w, L=n} = \pr{\cT^w(n)=y, \cT_-^w\cap (K\cup B)=\emptyset, \cT^w[0,n)\cap K=\emptyset}.
	\]
	Taking now the sum over all $y\in K'$ and all $n\geq 1$ completes the proof. 
\end{proof}

For a finite subset $K\subset \mathbb Z^d$, and $A\subseteq \mathbb Z^d$, we write $e_K(A) = \sum_{x\in A} e_K(x)$. 
\begin{lemma}\label{lem-shrink}
\rm{For any finite set $K\subseteq \Z^d$ containing two disjoint sets $U$ and $V$, we have
\begin{equation*}
e_K(U\cup V)\ge e_{K\bs V}(U).
\end{equation*}}
\end{lemma}
\begin{remark}\label{rem.inclusion} \rm{An immediate consequence of this lemma is that the branching capacity is monotone for inclusion, a fact already proved by~\cite{Zhu1}. Indeed, applying the lemma with $V= K\bs U$, gives that if $U\subseteq K$, then $\cpc{U} \le \cpc{K}$. }
\end{remark} 

\begin{proof}[\bf Proof of Lemma~\ref{lem-shrink}]
Let $R>0$, be such that $K\subset B(0,R-1)$. Applying Proposition~\ref{pro:lastpassagedec} we get 
\begin{align*}
	e_K(U\cup V) = \sum_{w\in \partial B(0,R)} 
\mathbb P\big(\cTp^w \cap B(0,R)= \emptyset , \cT^w_+ \text{ first hits $K$ in }U\cup V\big).
\end{align*}
Observing that  $\{ \cT_+^w \text{ first hits } K\bs V\text{ in } U\}\subseteq \{ \cT_+^w \text{ first hits } K\  \text{in } U\cup V\} $, we then obtain
\begin{align*}
	e_K(U\cup V) \geq \sum_{w\in \partial B(0,R)} 
\mathbb P\big(\cTp^w \cap B(0,R)= \emptyset ,  \cT_+^w \text{ first hits } K\bs V\text{ in } U\big) = e_{K\setminus V}(U),
\end{align*}
applying Proposition~\ref{pro:lastpassagedec} again for the last equality.
\end{proof}

The next result provides the exact order of magnitude of the branching capacity of balls. 
\begin{proposition}[\cite{Zhu1}]\label{lem.bcapballs}
\rm{There exist positive constants $c_1,c_2$, such that for all $r\ge 1$, 
$$c_1\cdot r^{d-4}\le \cpc{B(0,r)} \le c_2\cdot r^{d-4}.$$}
\end{proposition}
We shall often use later, without further reference that for a set $K\subseteq B(0,R)$, one has $\cpc{K} \lesssim R^{d-4}$, which follows from a combination of the last two results.


\subsection{Hitting probability for a branching random walk} \label{sec.hit}

\begin{definition}
	\rm{
	Suppose that $S$ is a random walk indexed by a spatial rooted tree $T$. On the event that $S$ hits a set $A$ we define the first entry vertex to $A$ as the smallest vertex $u\in T$ in the lexicographical order for which $S_u\in A$. If the unique path from the root of $\cT$ to the first entry vertex of $A$ is given by $(v_0,v_1,\ldots,v_k)$ for some~$k\in \N$, then we set $\Gamma(S)=(S_{v_0},\ldots, S_{v_k})$. We say that $S$ hits $A$ via $\gamma$, if $\Gamma(S)=\gamma$. 
	We also say that $S$ first hits the set $A$ in $a\in A$ if at the first entry vertex to $A$ the walk $S$ is at $a$. }
\end{definition}

Recall that $\til{\cT^x_c}$ denotes the range of an adjoint branching random walk starting from $x$. 
For a set $A\subseteq \mathbb Z^d$, we write
\[
b_A(x) = \mathbb P\big(\til{\cT}_c^x\cap A=\emptyset\big).
\]
For a path $\gamma:\{0,\dots,N\} \to \mathbb Z^d$ we write $|\gamma|=N$, i.e.\ $|\gamma|$ is the length of $\gamma$ without its first point. We write $s(\gamma)$ for the probability that a simple random walk started from $\gamma(0)$ follows this path for its first $|\gamma|$ steps. We say that $\gamma$ starts from $x$ if $\gamma(0)=x$, and that it goes from $x$ to a set $A$ and write $\gamma: x\to A$, if in addition $\gamma(N) \in A$ and $\gamma(\ell) \notin A$ for all $\ell < N$. Given $x,y\in \mathbb Z^d$, we write $\gamma:x\to y$ if $\gamma(0) = x$ and $\gamma(N)= y$. 

\begin{proposition}[{\cite[Proposition~5.1]{Zhu1}}]\label{prop.hit.Zhu}
\rm{Let $A\subset \mathbb Z^d$ and $x\in A^c$. Then for any $\gamma:x\to A$ we have 
$$\mathbb P(\cT^x_c \text{ hits $A$ via $\gamma$}) = s(\gamma) \prod_{\ell=0}^{|\gamma|-1} b_A(\gamma(\ell)). $$
} 
\end{proposition}

We now recall some hitting probability estimates obtained by Zhu. Given $x\in \mathbb Z^d$ and $K\subset \mathbb Z^d$, we let $d(x,K) = \inf \{\|x-y\| :y\in K\}$, and $\text{diam}(K) = \sup\{\|x-y\|:x,y\in K\}$.

\begin{theorem}[\cite{Zhu1,Zhu2}]\label{theo-zhu}
\rm{Let $\varepsilon>0$ be fixed. There exist positive constants $c_1,c_2$, such that for any finite nonempty $K\subset \Z^d$ and any 
$x\in \Z^d$, with $d(x,K)\ge \varepsilon \cdot \textrm{diam}(K)$, one has
\begin{equation}\label{hitting-infinite}
c_1 \frac{\cpc{K}}{d(x,K)^{d-4}}\le \mathbb P\big(\cTp^x\cap K\neq \emptyset \big)\le
c_2 \frac{\cpc{K}}{d(x,K)^{d-4}},
\end{equation}
and 
\begin{equation}\label{hitting-critical}
c_1 \frac{\cpc{K}}{d(x,K)^{d-2}}\le \mathbb P\big(\cTc^x\cap K\neq \emptyset\big)\le
c_2 \frac{\cpc{K}}{d(x,K)^{d-2}}.
\end{equation}
}
\end{theorem}

\begin{remark}\label{rem:boundforadjoint}
\rm{We note that the same estimate as~\eqref{hitting-critical} holds as well for~$\til{\cT}_c^x$. 
}	
\end{remark}


\section{Upper bound on $\|Ge_K\|_\infty$}\label{sec.Proof1}

In this section we give the proof of the first part of Theorem~\ref{theo.potentiel}, namely~\eqref{Upper.theo}. 
Using~\eqref{Upper.theo} we then give a proof of Corollary~\ref{thm:variationalchar} assuming a finite third moment on $\mu$.  
We start with a sketch of the proof in Section~\ref{sec-sketch}.
Then after recalling some standard estimates for a simple random walk in Section~\ref{subsec.RW}, we state and prove some results on hitting times for branching random walks 
in  Section~\ref{subsec.hitBRW}.  The proofs of~\eqref{Upper.theo} and Corollary~\ref{thm:variationalchar} are 
deferred to Sections~\ref{subsec.upper} and~\ref{subsec.cor.var} respectively.

\subsection{Sketch of proof of~\eqref{Upper.theo}}\label{sec-sketch}

First, since we seek an upper bound of $Ge_K(x)$, which is uniform in $K$ and $x$, one can always assume that $x$ is at the origin. 
Now by decomposing $K$ into slices $K=\cup_i K_i$, where $K_i = K\cap \{2^i \le \|x\|\le 2^{i+1}\}$, we find that 
$$Ge_K(0) \asymp\sum_i G(2^i) e_K(K_i),$$
with $e_K(K_i) = \sum_{y\in K_i} e_K(y)$. Thus one needs to estimate $e_K(K_i)$ now. For each $y\in K_i$, we decompose the event $\{\mathcal T^y_- \cap K=\emptyset\}$ according to the last point on $\partial B(0,2^{i+2})$ visited by the spine, and using shift invariance of the tree we arrive at 
$$e_K(K_i) \le \sum_{w\in \partial B(0,2^{i+2})} \mathbb P\big(\mathcal T^w_+\cap K_i\neq \emptyset, \mathcal T_-^w \cap (B_{i+2}\cup K) = \emptyset\big),$$
see~\eqref{eq:equilprobki} below. If the past and future trees were independent, we could separate both events in the probability on the right-hand side. Now for any $w\in \partial B(0,2^{i+2})$, by~\eqref{hitting-infinite} (which holds as well for $\mathcal T_+^w$), one has 
$$\mathbb P\big(\mathcal T^w_+\cap K_i\neq \emptyset)\asymp G(2^i) \cdot \textrm{BCap}(K_i),$$ 
while by Proposition~\ref{lem.bcapballs}, 
$$\sum_{w\in \partial B(0,2^{i+2})} \mathbb P\big(\mathcal T^w_-\cap B_{i+2}= \emptyset\big) = \textrm{BCap}\big(B(0,2^{i+2})\big) \lesssim \frac 1{G(2^i)}.$$
Moreover, at a heuristic level the events $\{\mathcal T^w_- \cap K_j=\emptyset\}$, for $j\ge i+3$, can be considered as being almost independent, and also independent of the event $\{\mathcal T^w_-\cap B_{i+2}= \emptyset\}$, 
since they depend on different pieces of $\mathcal T_-^w$ involving different scales. Thus we may infer that for some constant $c>0$,  
$$\mathbb P\big(\mathcal T_-^w \cap (B_{i+2}\cup K) = \emptyset\big) \lesssim \frac 1{G(2^i)} \cdot \prod_{j\ge i+3} \Big(1 - c\cdot G(2^j) \textrm{BCap}(K_j)\Big). $$ 
Therefore, under these rough independence assumptions, we arrive at 
\begin{equation}\label{goal.sketch}
G(2^i) e_K(K_i) \lesssim G(2^i) \textrm{BCap}(K_i) \cdot \exp\Big(-c\sum_{i+3\le j\le I} G(2^j) \textrm{BCap}(K_j)\Big),
\end{equation}
where $I$ is the maximal index $i$ such that $K_i$ is nonempty, and we conclude by observing that for any sequence $(\varepsilon_i)_{i\ge 0}$,  bounded by one,  
$$\sup_{I\ge 1}\, \sum_{i\le I} \varepsilon_i \cdot \exp\Big(-c \sum_{j\ge i+3} \varepsilon_j\Big), $$
is bounded by a constant that does not depend on the sequence $(\varepsilon_i)_{i\ge 1}$. Now of course, while we will prove that~\eqref{goal.sketch} is indeed correct, 
the whole technical matter of the proof is to deal with the fact that the events above are not really independent. 
In particular it is only once we condition on the positions of the walk on the spine that the past and future can be decorrelated. However, to make the arguments work fine, 
one also needs to ensure that one can place ourselves on the typical event, when the spine spends a time of order $2^{2j}$ in each of the sets $\{2^j\le \|x\|\le 2^{j+1}\}$, which leads to some technical difficulties.


\subsection{Simple random walk estimates}\label{subsec.RW}
We collect here a number of preliminary estimates concerning the simple random walk on $\mathbb Z^d$, that will be used for the proof of~\eqref{Upper.theo}. We write $\mathbb P_x$ for the law of a simple random walk started from $x\in \Z^d$ and $\mathbb E_x$ for the corresponding expectation. 
%
We let $X$ be a simple random walk in $\Z^d$. For $r>0$, we let $\tau_r$ be the first hitting time of $\partial B(0,r)$ and~$\tau_r^+$ the first return time to $B(0,r)$, i.e.\
\[
\tau_r = \inf\{t\geq 0: X_t\in \partial B(0,r)\} \quad \text{ and } \quad \tau_r^+ = \inf\{ t\geq 1: X_t\in \partial B(0,r)\}.
\]

\begin{lemma}\label{cl:exittime}
\rm{
Let $\delta>0$. There exists a positive constant $c=c(\delta)$ such that if  $R>0$, then 
for all $u\in \partial B(0,R)$	we have 
	\[
	\estart{\tau_{R(1+\delta)}\cdot \1(\tau_{R(1+\delta)}<\tau_R^+)}{u} \leq c R. 
	\]
	}
\end{lemma}

\begin{proof}[\bf Proof]

 Let $\tau = \tau_{R(1+\delta)}\wedge \tau_R$. Applying the optional stopping theorem to the martingale $(\norm{X_n}^2 -n)_{n\geq 0}$ we obtain that there exists a positive constant $c=c(\delta)$ such that 
\[
\estart{\tau_{R(1+\delta)}\cdot \1(\tau_{R(1+\delta)}<\tau_R^+)}{u} \leq 1+ \sup_{\substack{v\notin B(0,R)\\ v\sim u}}\estart{\tau}{v} \leq c R
\]
and this concludes the proof.
\end{proof}

\begin{lemma}\label{cl:estimforsrw}
\rm{
There exists $c>0$ such that the following holds. Let $R>0$, $u\in \partial B(0,R)$ and $v\in \partial B(0,2R)$. Then 
\[
\prstart{X_{\tau_R}=u, \tau_R<cR^2}{v} \leq \frac{1}{2}\cdot \prstart{X_{\tau_R}=u}{v}.
\]
}
\end{lemma}
\begin{proof}[\bf Proof]
Let $\tau=\inf\{t\geq 0: X_t\in \partial B(0,3R/2)\}$ be the first hitting time of $\partial B(0,3R/2)$. Let $c$ be a positive constant to be determined later. By the strong Markov property applied to $\tau$ we have 
\begin{align*}
	\prstart{X_{\tau_R}=u, \tau_R<cR^2}{v} &= \sum_{w\in \partial B(0,3R/2)} \sum_{n\leq \lfloor cR^2\rfloor}\sum_{s<n} \prstart{X_s=w, \tau=s}{v} \prstart{X_{\tau_R}=u, \tau_R=n-s}{w} \\ 
	&\leq  \sum_{w\in \partial B(0,3R/2)} \sum_{s<\lfloor cR^2\rfloor} \prstart{X_s=w, \tau=s}{v} \prstart{X_{\tau_R}=u}{w}.
\end{align*}
	By a proof similar to~\cite[Lemma~6.3.7]{LL10}, we get that there exist universal constants $c_1,c_2$ so that for all $w\in \partial B(0,3R/2)$ we have 	\begin{align}\label{eq:uptoconst}
	\prstart{X_{\tau_R}=u}{w}\leq 
	\frac{c_2}{R^{d-1}} \quad \text{ and } \quad \prstart{X_{\tau_R}=u}{v}\geq 
	\frac{c_1}{R^{d-1}}.
	\end{align}
	Therefore, plugging this above gives 
	\begin{align*}
	\prstart{X_{\tau_R}=u, \tau_R<cR^2}{v} \leq \frac{c_2}{R^{d-1}} \cdot \prstart{\tau<\lfloor cR^2\rfloor}{v}.
	\end{align*}
	On the event $\{\tau<\lfloor cR^2\rfloor\}$, the walk must travel distance $R/2$ in time less than $\lfloor cR^2\rfloor$, and hence this has probability less than $\exp(-c_3/c)$, where $c_3$ is a positive constant. Plugging this  above gives 
	\begin{align*}
		\prstart{X_{\tau_R}=u, \tau_R<cR^2}{v} \leq \frac{c_2}{R^{d-1}} \cdot \exp(-c_3/c).	\end{align*}
	Taking now $c$ sufficiently small so that $c_2\cdot \exp(-c_3/c)\leq c_1/2$ and using~\eqref{eq:uptoconst} we conclude
	\[
	\prstart{X_{\tau_R}=u, \tau_R<cR^2}{v} \leq \frac{1}{2}\cdot\frac{c_1}{R^{d-1}}\leq \frac{1}{2}\cdot \prstart{X_{\tau_R}=u}{v}, 
	\]
	and this finishes the proof.
\end{proof}

\begin{lemma}\label{cl:length}
	\rm{There exists a positive constant $c$ so that the following holds. Let $R>0$, $u\in \partial B(0,R)$ and $v\in \partial B(0,2R)$. Then we have 
	\[
	\sum_{n\ge 0} n\cdot  \prstart{X_n=v, \tau_R^+>n}{u} \leq c\cdot  R^2 \cdot \sum_{n\ge 0} \prstart{X_n=v, \tau_R^+>n}{u}.
		\]
		}
\end{lemma}
\begin{proof}[\bf Proof]
First of all using a reversal argument we see that 
\[
\sum_{n\geq 0} \prstart{X_n=v, \tau_R^+>n}{u}= \prstart{X_{\tau_R}=u}{v}.
\]
Since $\prstart{X_{\tau_R}=u}{v}\asymp R^{1-d}$ (see~\cite[Lemma~6.3.7]{LL10}), to prove the claim it suffices to show that there exists a positive constant $c$ so that 
\begin{align}\label{eq:goalforclaim}
\sum_{n} n\cdot  \prstart{X_n=v, \tau_R^+>n}{u} \leq \frac{c}{R^{d-3}}.
\end{align}
Let $\tau$ be the first hitting time of $\partial B(0,3R/2)$. Then by the strong Markov property applied to $\tau$ we have 
\begin{align*}
	\sum_{n\geq 0} &n\cdot  \prstart{X_n=v, \tau_R^+>n}{u}  \\&= \sum_{n\geq 0} \sum_{w\in \partial B(0,3R/2)}\sum_{s<n} n \cdot \prstart{\tau=s, X_s=w, \tau<\tau_R^+}{u} \prstart{X_{n-s}=v, \tau_R>n-s}{w} \\
	&=\sum_{w\in \partial B(0,3R/2)} \sum_{s\geq 0} \prstart{\tau=s, X_s=w, \tau<\tau_R^+}{u}\left(\sum_{n>s} n\cdot \prstart{X_{n-s}=v, \tau_R>n-s}{w} \right)\\
	&\leq \sum_{w\in \partial B(0,3R/2)} \sum_{s\geq 0} \prstart{\tau=s, X_s=w, \tau<\tau_R^+}{u}\left(\sum_{n>s} n\cdot \prstart{X_{n-s}=v}{w} \right).
\end{align*}
Using the local central limit theorem we now obtain
\[
\sum_{n>s} n\cdot \prstart{X_{n-s}=v}{w} = \sum_{n=1}^{\infty} n\cdot   \prstart{X_{n}=v}{w} + s \cdot \sum_{n=1}^{\infty} \prstart{X_{n}=v}{w} \asymp \frac{1}{R^{d-4}} + \frac{s}{R^{d-2}} .
\] 
Plugging this above we deduce
\begin{align*}
	\sum_{n\geq 0} n\cdot  \prstart{X_n=v, \tau_R^+>n}{u}  &\lesssim \sum_{w\in \partial B(0,3R/2)} \sum_{s \geq 0}\prstart{\tau=s, X_s=w, \tau<\tau_R^+}{u} \cdot \left( \frac{1}{R^{d-4}} + \frac{s}{R^{d-2}}\right) \\ &=
	\frac{1}{R^{d-4}} \cdot \prstart{\tau<\tau_R^+}{u} + \frac{1}{R^{d-2}}\cdot \estart{\tau\cdot \1(\tau<\tau_R^+)}{u} \lesssim \frac{1}{R^{d-3}},
\end{align*}
where we used that Lemma~\ref{cl:exittime} for the last inequality. This now concludes the proof.
\end{proof}


\subsection{Hitting times for branching random walks}\label{subsec.hitBRW}

For a path $\gamma :\{0,\dots,N\} \to \mathbb Z^d$, and $A\subseteq \mathbb Z^d$, we write 
$\gamma \subseteq A$, if $\gamma(\ell) \in A$, for all $\ell \ge 1$ (note that we allow the starting point of the path to be in $A^c$). Recall the notation introduced at the beginning of Section~\ref{sec.hit}.

\begin{lemma}\label{lem:avoidprob}
\rm{	There exists a positive constant $c$ so that for any $R\ge 1$ any $K\subseteq B(0,R/2)$, any $u\in \partial B(0,R)$ and $v\in \partial B(0,2R)$, one has 
	\[
	\sum_{\substack{\gamma: u\to v\\ \gamma\subseteq B(0,R)^c}} s(\gamma) \prod_{\ell=0}^{|\gamma|} b_{K}(\gamma(\ell)) \leq \exp\left(-c\frac{\cpc{K}}{R^{d-4}} \right)
\cdot \sum_{\substack{\gamma: u\to v\\ \gamma\subseteq B(0,R)^c}} s(\gamma).	\]
}
\end{lemma}
\begin{proof}[\bf Proof]
Since $\gamma\subseteq B(0,R)^c$ and $K\subseteq B(0,R/2)$, using~\eqref{hitting-critical} and Remark~\ref{rem:boundforadjoint} we get that for a positive constant $c$ and for all~$\ell\ge 0$
\[
b_K(\gamma(\ell)) \leq 1-c\frac{\cpc{K}}{R^{d-2}}.
\]
	So taking the product over all $\ell$ we get for a positive constant $c'$
	\begin{align*}
		\prod_{\ell=0}^{|\gamma|}b_K(\gamma(\ell)) \leq \exp\left(-c'|\gamma| \frac{\cpc{K}}{R^{d-2}}\right).
	\end{align*}
	Thus we deduce 
	\begin{align*}
		\sum_{\substack{\gamma: u\to v\\ \gamma\subseteq B(0,R)^c}} s(\gamma)\prod_{\ell=0}^{|\gamma|}b_K(\gamma(\ell)) \leq \sum_{\substack{\gamma: u\to v\\ \gamma\subseteq B(0,R)^c}} s(\gamma) \exp\left(-c'|\gamma| \frac{\cpc{K}}{R^{d-2}}\right).
	\end{align*}
	Let $c_1$ be a positive constant to be determined later. We now show that the last sum above is upper bounded by the sum where we restrict $\gamma$ to have length at least $c_1R^2$. First we write 
	\begin{align*}
		&\sum_{\substack{\gamma: u\to v\\ \gamma\subseteq B(0,R)^c}} s(\gamma) \exp\left(-c'|\gamma| \frac{\cpc{K}}{R^{d-2}}\right) 
		&\leq \sum_{\substack{\gamma: u\to v\\ \gamma\subseteq B(0,R)^c\\ |\gamma|\leq c_1R^2}} s(\gamma) + \sum_{\substack{\gamma: u\to v\\ \gamma\subseteq B(0,R)^c\\ |\gamma|>c_1R^2}} s(\gamma) \exp\left(-c'c_1\frac{\cpc{K}}{R^{d-4}}\right).
		\end{align*}
	Since $\cpc{K}\leq C R^{d-4}$ for a positive constant $C$ we can take $c_1$ sufficiently small so that
			\[
	\exp\left(-c'c_1\frac{\cpc{K}}{R^{d-4}}\right) \leq 1-\frac{c'c_1}{2}\cdot \frac{\cpc{K}}{R^{d-4}},
	\]
and hence plugging this above we obtain
\begin{align*}
	&\sum_{\substack{\gamma: u\to v\\ \gamma\subseteq B(0,R)^c}} s(\gamma) \exp\left(-c'|\gamma| \frac{\cpc{K}}{R^{d-2}}\right)\\
	&\leq  \left( 1-\frac{c'c_1\cpc{K}}{2R^{d-4}} \right) \cdot \sum_{\substack{\gamma: u\to v\\ \gamma\subseteq B(0,R)^c}} s(\gamma) + \frac{c'c_1\cpc{K}}{2R^{d-4}}  \cdot \sum_{\substack{\gamma: u\to v\\ \gamma\subseteq B(0,R)^c\\ |\gamma|\leq c_1R^2}} s(\gamma).
\end{align*}
		Taking $c_1$ even smaller, applying a time reversal and using Lemma~\ref{cl:estimforsrw} we upper bound the quantity above by 
		\begin{align*}
\left( 1-  \frac{c'c_1\cpc{K}}{4R^{d-4}} \right)   \cdot \sum_{\substack{\gamma: u\to v\\ \gamma\subseteq B(0,R)^c}} s(\gamma) \leq \exp\left(-c'c_1\frac{\cpc{K}}{4R^{d-4}}\right) \cdot \sum_{\substack{\gamma: u\to v\\ \gamma\subseteq B(0,R)^c}} s(\gamma),
		\end{align*}
		and this concludes the proof. 			\end{proof}

\begin{lemma}\label{lem:exittime}
	\rm{
	For each $R\geq 1$ let $\tau_R^u$ be the generation of the first hitting vertex of $B(0,R)$ by $\cT^u_c$.
		There exists $C>0$, such that, for any $R\ge 1$,
	 and any $u\in \partial B(0,2R)$, we have 
	\[
	\estart{\tau_R^u\cdot \1(\tau_R^u<\infty)}{} \leq C.
	\]
	}
\end{lemma}

\begin{proof}[\bf Proof]
	We have 
	\begin{align}\label{eq:expectation}
\estart{\tau_R^u\cdot \1(\tau_R^u<\infty)}{}   \leq 2R^2\cdot \pr{\tau_R^u<\infty} +R^2\cdot \sum_{k\geq 2} \prstart{\tau_{R}^u\geq kR^2, \tau_R^u<\infty}{}.
	\end{align}
	Using Theorem~\ref{theo-zhu} we get 
	\[
	\pr{\tau_R^u<\infty} \lesssim \frac{1}{R^2},
	\]
	and hence it only remains to bound the sum appearing on the right hand side of~\eqref{eq:expectation}.
	We now notice that on the event $\{\tau_{R}^u>kR^2, \tau_R^u<\infty\}$, there must exist a node on the tree in generation $(k-1)R^2$, whose position is  not in $B(0,R)$, but which has at least one descendant at distance at least $R^2$ from it, whose position is in $B(0,R)$. Writing $Z_m$ for the number of vertices in generation $m$ and $\cS_\ell=B(0,R_{\ell+1})\setminus B(0,R_{\ell})$ with $R_\ell= 2^\ell R$, for $\ell\in \N$, we get
	\begin{align*}\label{eq:expectation2}
		\prstart{\tau_{R}^u\geq kR^2, \tau_R^u<\infty}{} \leq &\E{Z_{(k-1)R^2}}\\ 
&\nonumber		\times   \sum_{m=0}^{\infty} \prstart{X_{(k-1)R^2}\in \cS_m}{u} \max_{x\in \cS_m}\pr{\cT_c^x\text{ hits } B(0,R) \text{ after $R^2$ generations}},
		\end{align*}
		where $X$ is a simple random walk.
	For $m=0$ we bound the second probability appearing above by the probability that the tree survives for $R^2$ generations, which is of order $1/R^2$, by Kolmogorov's estimate again. For $m\geq 1$ we bound the second probability by $R^{d-4}/R_m^{d-2}$ using~\eqref{hitting-critical} and Proposition~\ref{lem.bcapballs}. We also have for a constant $c>0$, by the local central limit theorem, for all $k\geq 2$
	\[
	\prstart{X_{(k-1)R^2}\in \cS_m}{u} \lesssim \frac{R_m^d}{k^{d/2}R^d} \wedge \exp\left(-c\frac{R_m^2}{kR^2}\right).
	\]
	Putting everything together we get 
	\begin{align*}
\sum_{m=1}^{\infty} &\prstart{X_{(k-1)R^2}\in \cS_m}{u} \max_{x\in \cS_m}\prstart{\cT^x_c\text{ hits } B(0,R) \text{ after $R^2$ generations}}{} 
\\ &\lesssim \frac{1}{k^{d/2}}\cdot \frac{1}{R^2} + \sum_{m\geq 2: R_m^2\leq k R^2} \frac{R_m^d}{k^{d/2}R^d} \cdot \frac{R^{d-4}}{R_m^{d-2}} + \sum_{m: R_m^2>kR^2}\exp\left(-c\frac{R_m^2}{kR^2}\right) \cdot \frac{R^{d-4}}{R_m^{d-2}}  \asymp  \frac{1}{R^2} \cdot \frac{1}{k^{d/2-1}}. 
		\end{align*}
	Plugging this back into the sum in~\eqref{eq:expectation} and using that $\mathbb E[Z_{(k-1)R^2}] =1$ by criticality of $\mu$, concludes the proof.
\end{proof}


\subsection{Proof of~\eqref{Upper.theo}}\label{subsec.upper}
We prove here the first part of Theorem~\ref{theo.potentiel}. 
We let $R_0=0$ and $R_i=2^i$ for $i\geq 1$, as well as 
	 \[
	 B_i=B(x,R_i), \quad  \cS_i=\{y: R_i\leq \norm{y-x}\leq R_{i+1}\} \quad \text{ and }\quad K_i=K\cap \cS_i.
	 \]
	Let $I$ be the maximal index $i$ such that $K\cap \cS_i\neq \emptyset$. Then we have 
\begin{align}\label{eq:shells}
	\sum_{y\in K} G(x,y) e_K(y) \asymp \sum_{i=0}^{I} G(R_i) \sum_{y\in K_i} e_K(y).
\end{align}
Fix some $0\le i \le I$. Then applying Proposition~\ref{pro:lastpassagedec} with $K'=K_i$ and $B=B_{i+2}$ we obtain
\begin{align}\label{eq:equilprobki}
\sum_{y\in K_i}e_K(y)\leq \sum_{w\in \partial B_{i+2}}	\prstart{\cT_-^w\cap (B_{i+2}\cup K)=\emptyset, \cT_+^w\cap K_i\neq \emptyset}{}.
	\end{align}
		We now use the natural parametrisation of the spine and let $X^w$ be the simple random walk that the spine performs starting from $w$. Recall that we also write 
	$\F_-^w[a,b]$ (respectively $\F_+^w[a,b]$) for the positions of the walk indexed by $\cT$ starting from $w$ at vertices lying in the forest consisting of the adjoint trees in the past (respectively future)  emanating from the points on the spine with time index (in the natural parametrisation of the spine) in $[a,b]$.
	
	Fix some $w\in \partial B_{i+2}$, and for each $j\ge i+2$ we let $\sigma_j$ be the last time that $X^w$ is on $\partial B_j$, i.e.
	\[
	\sigma_j=\sup\{n\geq 0: X^w(n)\in \partial B_j  \}.
	\]
	Note in particular that $\sigma_{i+2} = 0$ on the event $\{\cT^w_- \cap B_{i+2} = \emptyset\}$. 
	\begin{figure}[ht!]\label{fig:pastandfuture}
	\begin{center}
		\includegraphics[scale=1]{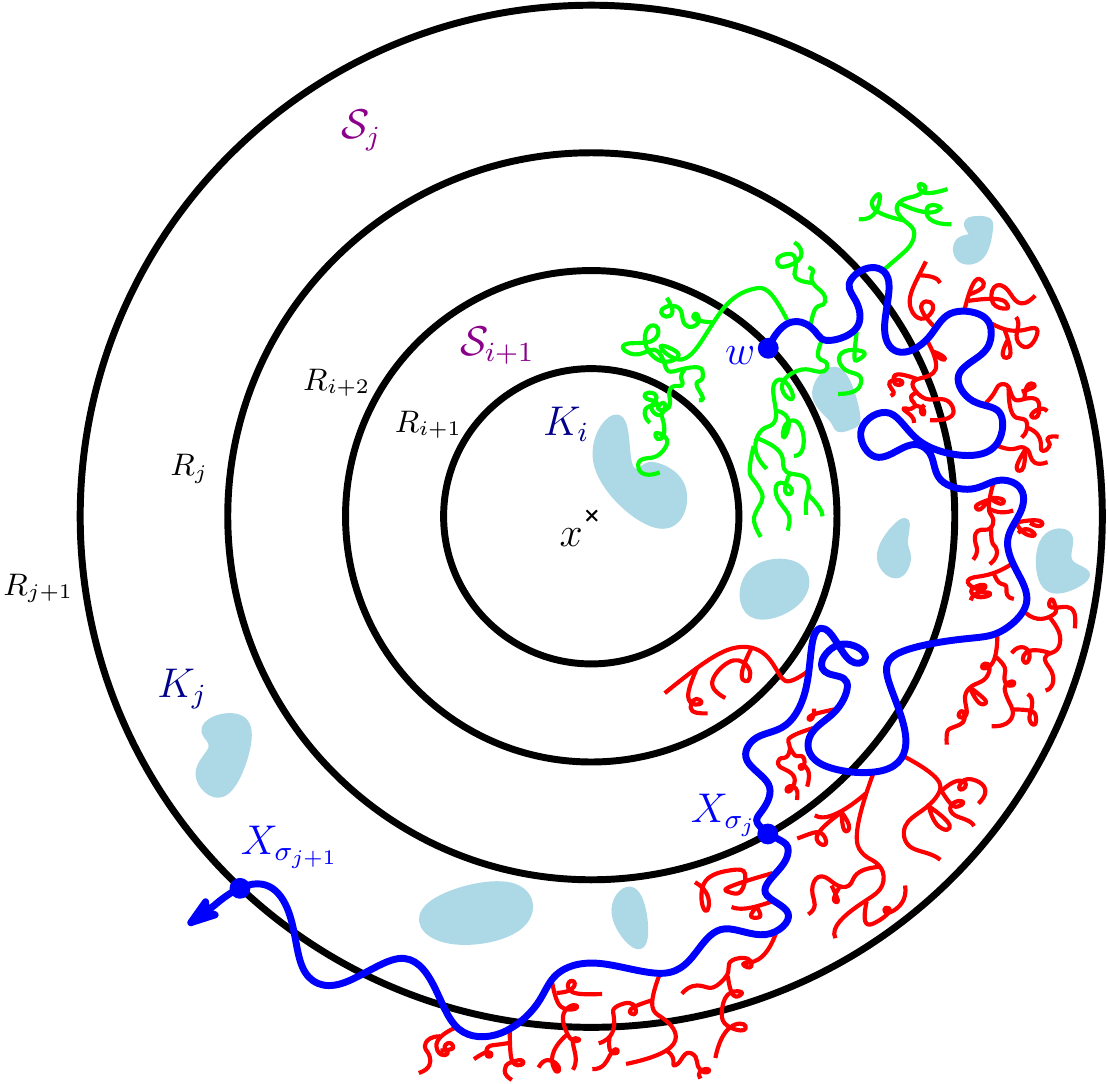}
		\end{center}
		\caption{The re-rooted tree at $w$ whose spine (in blue) has to avoid $B_{i+2}$ and the past trees (in red) have to avoid $K_j$ in the shell $\mathcal{S}_j$, whereas the future trees (in green) have to hit $K_i$.}
		\end{figure}
For $i,j\in \N$ with $j\geq i+2$ we 	define
\[
	A_{i,j} = \{\F_+^w[\sigma_j,\sigma_{j+1}]\cap K_i\neq \emptyset, \, \cT^w_- \cap B_{i+2} = \emptyset, \text{ and }\forall \ i+2\leq m\leq I+2, \ \F_-^w[\sigma_m,\sigma_{m+1}]\cap K_{m-2}=\emptyset\}.
	\]
	With this definition we then get 
	\begin{align}\label{eq:boundbyaij}
		\{\cT_-^w\cap (B_{i+2}\cup K)=\emptyset, \cT_+^w\cap K_i\neq \emptyset \} 
		 \subseteq \bigcup_{j=i+2}^{\infty} A_{i,j}. 
	\end{align}
	We now upper bound $\pr{A_{i,j}}$ for all $i,j$ with $j\geq i+2$. 
	First we consider $j>i+2$. In this case we have 
	\begin{align*}
		\prstart{A_{i,j}}{} = \sum_{\substack{u_{i+3}\in \partial B_{i+3}}} \cdots 
		\sum_{u_{I+3}\in \partial B_{I+3}}\prstart{\cap_{i+3\leq m\leq I+3}\{X^w(\sigma_m)=u_m\},  A_{i,j}}{}.
	\end{align*}
		For any path $\gamma$ using~\eqref{hitting-critical} and a union bound we get 	\begin{align*}
		\prcond{\F_+^w[\sigma_j,\sigma_{j+1}]\cap K_i\neq \emptyset}{X^w[\sigma_j,\sigma_{j+1}]=\gamma}{}		 \lesssim |\gamma| \cdot \frac{\cpc{K_i}}{R_{j}^{d-2}}.
	\end{align*}
		Recall that $\til{\cT}^z_c$ denotes an adjoint branching random walk started from $z\in \Z^d$, and that for a finite set $A\subseteq \Z^d$ we write
	\[
	b_A(z)  = \mathbb P\big(\til{\cT}_c^z\cap A=\emptyset\big).
	\]
Also recall that we write~$s(\gamma)$ for the probability that a simple random walk of length $|\gamma|$ started from $\gamma(0)$ follows the path $\gamma$. Using the above together with the independence of the adjoint trees hanging off the spine conditionally on the values of the spine, we get 
	\begin{align*}
		\prstart{A_{i,j}}{} \lesssim 	& \sum_{u_i+3 \in \partial B_{i+3}}\cdots \sum_{u_{I+3}\in \partial B_{I+3}} \Big( \sum_{\substack{\gamma:w\to u_{i+3}\\ \gamma\subseteq B_{i+2}^c}} s(\gamma)\prod_{\ell=1}^{|\gamma|}b_{B_{i+2}}(\gamma(\ell)) \Big) \times \Big(\sum_{\substack{\gamma: u_j\to u_{j+1}\\ \gamma \subseteq B_j^c}}|\gamma| s(\gamma) \cdot \frac{\cpc{K_i}}{R_j^{d-2}}\Big)  \\ 
		&\times \prod_{\substack{i+3\leq m\leq  I+2\\ m\neq j}}\Big(\sum_{\substack{\gamma: u_m\to u_{m+1} \\ \gamma\subseteq B_m^c}} s(\gamma)\prod_{\ell=0}^{|\gamma|}b_{K_{m-2}}(\gamma(\ell)) \Big)\times \mathbb P_{u_{I+3}}(\tau_{B_{I+3}}^+=\infty),
	\end{align*}
	where $\tau_A^+$ stands for the first return time to $A$ by a simple random walk.
Using Lemma~\ref{cl:length} we get 
\begin{align}\label{eq:oneofthesums}
\sum_{\substack{\gamma: u_j\to u_{j+1}\\ \gamma \subseteq B_j^c}}|\gamma| s(\gamma) \lesssim R_j^2 \cdot \sum_{\substack{\gamma: u_j\to u_{j+1}\\ \gamma \subseteq B_j^c}} s(\gamma). 
\end{align}
For any path $\gamma:w\to u_{i+3}$ with $\gamma\subseteq B_{i+2}^c$ one has by Proposition~\ref{prop.hit.Zhu}
\[
s(\gamma)\prod_{\ell=1}^{|\gamma|}b_{B_{i+2}}(\gamma(\ell)) = \pr{\cT^{u_{i+3}}_c \text{ hits } B_{i+2} \text{ via } \gamma^\leftarrow},
\]
where $\gamma^\leftarrow$ denotes the reversal of $\gamma$. Therefore, using also Theorem~\ref{theo-zhu} we now get 
\begin{align*}
\sum_{w\in \partial B_{i+2}}\sum_{\substack{\gamma:w\to u_{i+3}\\ \gamma\subseteq B_{i+2}^c}} s(\gamma)\prod_{\ell=1}^{|\gamma|}b_{B_{i+2}}(\gamma(\ell)) = \pr{\cT^{u_{i+3}}_c \text{ hits } B_{i+2}} \asymp \frac{1}{R_i^2}.
\end{align*}
 Using this together with~\eqref{eq:oneofthesums} and Lemma~\ref{lem:avoidprob}, we finally get that in the case $j>i+2$
\begin{equation*}
	\sum_{w\in \partial B_{i+2}} \prstart{A_{i,j}}{} \lesssim \frac{\cpc{K_i}}{R_j^{d-4}} \cdot \frac{1}{R_i^2}\cdot  \exp\left(-c\sum_{m=i+1}^{I} \frac{\cpc{K_m }}{R_m^{d-4}} \right)\cdot\sum_{u\in \partial B_{i+3}} \mathbb P_u(\tau_{B_{i+3}}^+=\infty),
\end{equation*}
and thus using that the capacity of a ball $B(0,R)$ is of order $R^{d-2}$, for the usual notion of capacity, we get 
\begin{equation}\label{eq:casejlarge}
\sum_{w\in \partial B_{i+2}} \mathbb P(A_{i,j}) \lesssim \frac{\cpc{K_i}}{R_j^{d-4}} \cdot R_i^{d-4}\cdot  \exp\left(-c\sum_{m=i+1}^{I} \frac{\cpc{K_m }}{R_m^{d-4}} \right). 
\end{equation}
	We now bound the sum $\sum_{w\in \partial B_{i+2}}\pr{A_{i,j}}$  in the case when $j=i+2$.
	Conditionally on the event $\{X^w[\sigma_{i+2},\sigma_{i+3}]=\gamma\}$, we let $\cT_\ell$ be the $\ell$-th adjoint tree in the past hanging off $\gamma(\ell)$ (taking $\cT_0=\emptyset$) and $\cT_\ell'$ be the $\ell$-th adjoint tree in the future hanging off $\gamma(\ell)$. Then we get using again~\eqref{hitting-critical}, 
	\begin{align*}
		&\prcond{\F_+^w[0,\sigma_{i+3}]\cap K_i\neq \emptyset, \F_-^w[0,\sigma_{i+3}]\cap B_{i+2}=\emptyset}{X^w[0,\sigma_{i+3}]=\gamma}{} \\
		&\leq \sum_{\ell=0}^{|\gamma|}\pr{\cT'_\ell\cap K_i \neq \emptyset, \forall k\neq \ell, \cT_k\cap B_{i+2}=\emptyset}
	 \lesssim \sum_{\ell=0}^{|\gamma|}\frac{\cpc{K_i}}{R_i^{d-2}} \cdot \prod_{\substack{m=1\\ m\neq \ell}}^{|\gamma|} b_{B_{i+2}}(\gamma(m)). 	\end{align*}
	Note that $b_{B_{i+2}}(\gamma(m))\geq \widetilde\mu(0)$ for all $m$, since if the tree dies out immediately, the adjoint branching random walk cannot hit the set $B_{i+2}$. Using this we deduce 
	\begin{align*}
		\prcond{\F_+^w[0,\sigma_{i+3}]\cap K_i\neq \emptyset, \F_-^w[0,\sigma_{i+3}]\cap B_{i+2}=\emptyset}{X^w[0,\sigma_{i+3}]=\gamma}{} 
		\lesssim |\gamma| \cdot \frac{\cpc{K_i}}{R_i^{d-2}} \cdot \prod_{\ell=1}^{|\gamma|}b_{B_{i+2}}(\gamma(\ell)).
	\end{align*}
	Using this we obtain 
	\begin{align*}
		\sum_{w\in \partial B_{i+2}}\pr{A_{i,i+2}} \lesssim \sum_{w\in \partial B_{i+2}}\sum_{u_{i+3}\in \partial B_{i+3}}\cdots \sum_{u_{I+3}\in \partial B_{I+3}}
		\prod_{m=i+3}^{I+2} \Big(\sum_{\substack{\gamma: u_m\to u_{m+1}\\ \gamma \subseteq B_m^c}} s(\gamma) \prod_{\ell=0}^{|\gamma|} b_{K_{m-2}}(\gamma(\ell)) \Big) \\ 
		\times \Big( \sum_{\substack{\gamma: w\to u_{i+3} \\ \gamma\subseteq B_{i+2}^c}} s(\gamma)\cdot |\gamma| \cdot \frac{\cpc{K_i}}{R_i^{d-2}} \cdot \prod_{\ell=0}^{|\gamma|-1}b_{B_{i+2}}(\gamma(\ell))\Big) \times \prstart{\tau_{B_{I+3}}^+=\infty}{u_{I+3}}.
		\end{align*}
	Using Lemma~\ref{lem:avoidprob} again we get 
	\begin{align}\label{eq:casejsmall}
	\begin{split}
		\sum_{w\in \partial B_{i+2}}\pr{A_{i,j}} & \lesssim \sum_{\substack{w\in \partial B_{i+2}\\ u\in \partial B_{i+3}}}			
		 \exp\left(-c\sum_{m=i+1}^I \frac{\cpc{K_{m}}}{R_m^{d-4}} \right)\\ &\times  \Big( \sum_{\substack{\gamma: w\to u \\ \gamma\subseteq B_{i+2}^c}} s(\gamma)\cdot |\gamma| \cdot \frac{\cpc{K_i}}{R_i^{d-2}} \cdot \prod_{\ell=1}^{|\gamma|}b_{B_{i+2}}(\gamma(\ell))\Big)\cdot \prstart{\tau_{B_{i+3}}^+=\infty}{u}.
\end{split}
	\end{align}
	We note again that for any path $\gamma:w\to u$, with $\gamma\subseteq B_{i+2}^c$, one has by Proposition~\ref{prop.hit.Zhu}
	\begin{align*}
		s(\gamma)\prod_{\ell=1}^{|\gamma|}b_{B_{i+2}}(\gamma(\ell)) = \pr{\cT_c^u \text{ hits } B_{i+2} \text{ via } \gamma^{\leftarrow}},
	\end{align*}
	where $\gamma^\leftarrow$ denotes the reversal of $\gamma$. 
	Writing $\tau^u$ for the generation of the first hitting vertex of $B_{i+2}$ by a critical branching random walk started from~$u$, we therefore deduce by Lemma~\ref{lem:exittime}, that
	\begin{align*}
		\sum_{w\in \partial B_{i+2}}\sum_{\substack{\gamma: w\to u \\ \gamma\subseteq B_{i+2}^c}} s(\gamma)\cdot |\gamma|  \cdot \prod_{\ell=1}^{|\gamma|}b_{B_{i+2}}(\gamma(\ell)) = \E{\tau^u\cdot \1(\tau^u<\infty)}\leq C. 
	\end{align*}
		Plugging this back into~\eqref{eq:casejsmall} we get that for all $i,j$ with $j\geq i+2$ we have 
		\begin{equation*}
	\sum_{w\in \partial B_{i+2}}\prstart{A_{i,j}}{} \lesssim \frac{\cpc{K_i}}{R_j^{d-4}} \cdot R_i^{d-4}\cdot
	\exp\left(-c\sum_{m=i+1}^{I} \frac{\cpc{K_{m}}}{R_m^{d-4}} \right). 	
\end{equation*}
Then taking the sum over all $j\geq i+2$ and using~\eqref{eq:boundbyaij} and~\eqref{eq:equilprobki} yields
\begin{align*}
\sum_{y\in K_i}e_{K}(y)\leq \sum_{j\geq i+2}\sum_{w\in \partial B_{i+2}}\prstart{A_{i,j}}{} \lesssim \cpc{K_i} \cdot \exp\left(-c\sum_{m=i+1}^{I} 
\frac{\cpc{K_m}}{R_m^{d-4}} \right).
\end{align*}
We can now conclude the proof using~\eqref{eq:shells} to get 
\begin{align*}
	\sum_{y\in K} G(x,y) e_K(y) \lesssim \sum_{i=0}^{I} \frac{\cpc{K_i}}{R_i^{d-4}}\cdot  \exp\left( -c \sum_{m=i+1}^{I}\frac{\cpc{K_m}}{R_m^{d-4}} \right).
\end{align*}
Setting $\epsilon_i = \cpc{K_i}/R_i^{d-4}$ we know that there exists a universal constant $C>0$, such that $0<\epsilon_i\leq C$, for all $i$. We have now reduced the problem to proving that for such $(\epsilon_i)$ we have 
\[
	\sum_{i=0}^{I} \epsilon_i\cdot  \exp\left(-\sum_{m=i+1}^{I} \epsilon_m\right)\lesssim 1. 
	\]
Using that for $x\leq C$ we have $e^{-x}\leq 1-x/e^C$ we get 
	\begin{align*}
		\sum_{i=0}^{I} \epsilon_i\cdot  \exp\left(-\sum_{m=i+1}^{I} \epsilon_m\right) \leq \sum_{i=1}^{I} \epsilon_i \prod_{m=i+1}^{I}\left( 1-\frac{\epsilon_m}{e^C}\right).
	\end{align*}
Hence, it suffices to prove that for $(\epsilon_i)_i$ with $\epsilon_i\in [0,1]$ we have 
\[
	\sum_{i=0}^{I} \epsilon_i \prod_{m=i+1}^{I}\left( 1-\epsilon_m\right)\leq 1.
	\]
	It now suffices to prove that 
	\begin{align}\label{eq:sumoverepsiloni}
	\sum_{i=0}^{I} \epsilon_i \prod_{m=i+1}^{I}\left( 1-\epsilon_m\right) + \prod_{i=1}^{I}(1-\epsilon_i) =1.
	\end{align}
	Indeed, notice that for any $i$ we have 
	\[
	\epsilon_i\prod_{m=i+1}^{I}\left( 1-\epsilon_m\right) + \prod_{m=i}^{I}(1-\epsilon_m) = \prod_{m=i+1}^{I}(1-\epsilon_m).
	\]
	Applying this iteratively we deduce~\eqref{eq:sumoverepsiloni} and this concludes the proof of~\eqref{Upper.theo}.


\subsection{Variational characterisation: first proof of Corollary~\ref{thm:variationalchar}}\label{subsec.cor.var}
In this section we prove Corollary~\ref{thm:variationalchar}, using only the results obtained so far, but under the additional hypothesis that $\mu$ has a finite third moment. We start with a technical lemma that will be needed in the proof. 
\begin{lemma}\label{lem:technical}
\rm{	There exists a positive $C$ so that the following is true. For every $M>0$, all~$x,y \in B(0,M)$ and all $\rho$ with $\|\rho\|\geq 10M$, we have 
\begin{align*}
	\sum_{z\in \mathbb Z^d} g(x,z) g(z,\rho) G(y,z) \leq C \cdot G(\rho) \cdot G(x-y), \\ 
	\sum_{z\in \mathbb Z^d} g(x,z) g(y,z) G(\rho, z)  \leq C\cdot  G(\rho)\cdot G(x-y).
\end{align*}
}
\end{lemma}

\begin{proof}[\bf Proof]
We start with the first inequality. We consider the set $E=\{z: \norm{z-\rho}\leq \norm{\rho}/{2}\}$. Then for~$z\in E$ we have $\norm{y-z}, \norm{x-z}\asymp \norm{\rho}$, so we get using~\eqref{green.rw} and~\eqref{Green.asymp}, 
\begin{align*}
	\sum_{z\in E} g(x,z) g(\rho,z) G(z,y) &\asymp \frac{1}{\norm{\rho}^{2d-6}}\cdot \sum_{z\in E} \frac{1}{\norm{z-\rho}^{d-2}} \asymp \frac{1}{\norm{\rho}^{2d-8}} \lesssim \frac{1}{\norm{\rho}^{d-4}\norm{x-y}^{d-4}},
\end{align*}
where for the last step we used that $\norm{\rho}\geq \norm{x-y}$.
For the sum over $E^c$ we have 
	\begin{align}\label{eq:boundonic}
	\sum_{z\in E^c} g(x,z) g(\rho,z) G(z,y) \lesssim \frac{1}{\norm{\rho}^{d-2}}\cdot  \sum_{z} \frac{1}{\norm{z}^{d-2} \norm{z-(x-y)}^{d-4}}.
\end{align}
Setting $u=x-y$ and considering three different cases depending on whether $\norm{z - u}\leq \norm{u}/2$, $\norm{u}/2\leq \norm{z-u}\leq 2\norm{u}$ or $\norm{z-u}\geq 2\norm{u}$ we obtain that 
\begin{align*}
	\sum_{z} \frac{1}{\norm{z}^{d-2} \norm{z-(x-y)}^{d-4}} \lesssim \frac{1}{\norm{x-y}^{d-6}}.
\end{align*}
Plugging this into~\eqref{eq:boundonic} and using that $\norm{x-y}\leq \norm{\rho}$, we get the desired bound.

For the second sum of the statement considering again the set $E$ we have 
\begin{align*}
	\sum_{z\in E}g(x,z) g(y,z) G(\rho, z)  \asymp \frac{1}{\norm{\rho}^{2d-4}} \cdot \sum_{z\in E} \frac{1}{\norm{z-\rho}^{d-4}} \asymp \frac{1}{\norm{\rho}^{2d-8}} \lesssim \frac{1}{\norm{\rho}^{d-4}\norm{x-y}^{d-4}}.
\end{align*}
For the sum over $E^c$ we get 
\begin{align*}
	\sum_{z\in E^c} g(x,z) g(y,z) G(\rho, z) \lesssim \frac{1}{\norm{\rho}^{d-4}}\cdot \sum_{z} g(x,z)g(y,z) \asymp \frac{1}{\norm{\rho}^{d-4}}\cdot \frac{1}{\norm{x-y}^{d-4}}
\end{align*}
and this concludes the proof.
\end{proof}

	The proof of Corollary~\ref{thm:variationalchar} uses the same idea as the proof of~\cite[Theorem~2.2]{BPP95} (see also \cite[Proposition~8.26]{MP10}).

\begin{proof}[{\bf {Proof of Corollary~\ref{thm:variationalchar}}} when $\sum_i i^3 \mu(i)<\infty$]
	To prove the upper bound we take $\nu=e_K/\cpc{K}$ and use \eqref{Upper.theo}. 
	 
	To prove the lower bound, it suffices to show that for any probability measure $\nu$ supported on $K$ we have 
	\begin{align*}
		\cpc{K} \gtrsim \frac{1}{\sum_{x\in K}\sum_{y\in K} \nu(x)\nu(y) G(x,y)}. 
	\end{align*}
	Let $\nu$ be a probability measure on~$K$. Let~$\rho\in \Z^d$ and 
	\[
	Z=\sum_{u\in \cT_-} \sum_{y\in K} \frac{\1(S^\rho_u=y)\nu(y)}{G(\rho,y)}.
	\]
	Then it is clear that by the Payley-Zygmund inequality
	\begin{align}\label{eq:payley}
	\prstart{\cT^\rho_-\cap K\neq \emptyset}{} \geq  \prstart{Z>0}{}\geq \frac{(\estart{Z}{})^2}{\estart{Z^2}{}}.
	\end{align}
	 By Proposition~8.1 of~\cite{Zhu1} we have that  
	\[
	\lim_{\norm{\rho}\to\infty} \frac{\prstart{\cT_-^\rho\cap K\neq \emptyset}{}}{G(\rho)} = \cpc{K}.
	\]
	In view of this and~\eqref{eq:payley} it suffices to prove that for $\norm{\rho}$ sufficiently large
	\begin{align}\label{eq:goalforzandz2}
		\frac{\E{Z^2}}{(\E{Z})^2} \lesssim \frac{1}{G(\rho)}\cdot \sum_{x,y\in K}\nu(x)\nu(y) G(x,y).
	\end{align}
For the first moment of $Z$ we have 
	\begin{align}\label{eq:firstmomentz}
		\estart{Z}{} = \sum_{y\in K} \nu(y) \frac{G(\rho,y)}{G(\rho,y)} = 1.
	\end{align}
	We now turn to the second moment. For this we have
	\begin{align}\label{eq:seconmomentzfirstexp}
		\estart{Z^2}{} \leq 2 \estart{\sum_{k\leq n}\sum_{y\in K}\sum_{x\in K} \frac{\1(\cT^\rho(-k)=x) \1(\cT^\rho(-n)=y) \nu(x)\nu(y)}{G(\rho,x)G(\rho,y)} }{}.
	\end{align}
	Applying the shift $\theta^k$ and using the invariance of the tree and the walk under $\theta$ we obtain
	\begin{align*}
		\prstart{\cT^\rho(-k)=x, \cT^\rho(-n)=y}{} = \prstart{\cT^x(k)=\rho, \cT^x(-n+k)=y}{}.
	\end{align*}
Taking the sum over all $k$ and $n$ we obtain
\begin{align*}
	\estart{\sum_{k\leq n}\1(\cT^\rho(-k)=x)\1(\cT^\rho(-n)=y)}{} = \estart{\left(\sum_{k\geq 0}\1(\cT^x(k)=\rho)\right)\cdot \left(\sum_{n\geq 0}\1(\cT^x(-n)=y)\right)}{}. 
\end{align*}	
	So above we get the product of the number of visits to $\rho$ in the future of~$\cT$ with the number of visits to $y$ in the past of~$\cT$ including the root. Let $(X^x_m)_{m\in \N}$ be the random walk of the spine started from $x$ and denote by $V_f(k)$ the number of visits to $\rho$ in the adjoint tree in the future hanging off $X^x_k$ and $V_p(k)$ for the number of visits to $y$ in the adjoint tree in the past hanging off $X^x_k$. We then have, denoting 
	by $(\pm e_i)_{i=1,\dots,d}$ the neighbors of the origin in $\mathbb Z^d$, 
	\begin{align*}
		&\estart{\sum_{k\leq n}\1(\cT^\rho(-k)=x)\1(\cT^\rho(-n)=y)}{} \leq  \estart{\sum_{k,\ell\geq 0} V_p(k) V_f(\ell)}{} \\
		&\asymp \sum_{k\neq \ell} \E{g(X_k^x,y) g(X_\ell^x,\rho)} + \sum_{k\geq 0}\E{V_p(k)V_f(k)} \\ 
		&\lesssim \sum_{k\neq \ell} \E{g(X_k^x,y) g(X_\ell^x,\rho)}  + \sum_{i,j\geq 0} \mu(i+j+1)\cdot  i \cdot \E{ \max_{m\leq d}g(X_k^x\pm e_m, y)} \cdot j \cdot \E{\max_{m\leq d} g(X_k^x\pm e_m, \rho)} \\
		&\lesssim \sum_{k, \ell\geq 0}\E{g(X_k^x,y) g(X_\ell^x,\rho)} \leq \sum_{z,w}g(x,z)g(z,w)g(z,\rho)g(w,y) + \sum_{z,w} g(x,z) g(z,y)g(z,w)g(w,\rho) \\& \asymp \sum_z g(x,z)g(z,\rho)G(z,y) + \sum_z g(x,z)g(z,y)G(z,\rho), 
	\end{align*}
	where for the second line we used the independence between the trees in the past and future attached to different points on the spine, and for the fourth line we used the assumption that $\mu$ has a finite third moment.  
		Taking $\rho$ with $\norm{\rho}\geq 10\cdot \diam{K}$, and applying Lemma~\ref{lem:technical} we get 
	\begin{align*}
		\sum_z g(x,z)g(z,\rho)G(z,y) + \sum_z g(x,z)g(z,y)G(z,\rho)\lesssim G(\rho)\cdot G(x,y).
	\end{align*}
	Therefore for $\rho$ with $\norm{\rho}\geq 10\cdot \diam{K}$, plugging this in the above and then using~\eqref{eq:seconmomentzfirstexp} we deduce that 
	\[
	\E{Z^2} \lesssim \frac{1}{G(\rho)} \cdot \sum_{x,y \in K} \nu(x) \nu(y) G(x,y).
	\]
	This together with~\eqref{eq:firstmomentz} finish the proof of~\eqref{eq:goalforzandz2} and the proof of the theorem.
\end{proof}


\section{Moments of local times}\label{sec.Moment}

In this section we prove Theorem~\ref{theo.expmoment} and Corollary~\ref{cor.loctimeball}. To prove Theorem~\ref{theo.expmoment} we first show the result for a critical branching random walk, then for an adjoint 
critical branching random walk, and finally we extend it to the infinite tree-indexed random walk using the natural decomposition of $\cT$ 
in terms of a spine together with adjoint trees hanging off its vertices. We use here the notation $\{S_u\}_{u\in \cT}$ to denote the random walk indexed by $\cT$ and similarly with $\cT_c$ or $\widetilde \cT_c$. We then let $\mathbb P_x$ denote their law when the starting point is $x$, and write $\mathbb E_x$ for the corresponding expectation.  
 
We assume in the whole section that $\mu$ has a finite exponential moment, in other words we assume that there exists $\lambda>0$, such that 
$\sum_{i\ge 0} e^{\lambda i} \mu(i)<\infty$. We also assume that $\varphi:\mathbb Z^d\to [0,\infty)$ is a function satisfying $\|G\varphi\|_\infty\le 1$. 
Note that one may also assume that $\|g\varphi\|_\infty\le 1$ thanks to~\eqref{green.rw} and~\eqref{Green.asymp}.

\subsection{The case of a critical branching random walk}
We consider here the case of a critical branching random walk. 
The proof consists in bounding the moments using a suitable induction.

To start with let us bound the first moment. Recall that $g$ denotes the Green's function of the simple random walk $(X_n)_{n\ge 0}$, 
and notice that for any $x\in \mathbb Z^d$, one has with $Z_n$ the number of vertices in the $n$-th generation of $\cT_c$,   
\begin{equation}\label{first.moment}
\mathbb E_x\left[\sum_{v\in \cT_c} \varphi(S_v)\right] = \sum_{n\ge 0} \mathbb E[Z_n] \cdot \mathbb E_x[\varphi(X_n)] = g\varphi(x) \le 1, 
\end{equation}
using for the second equality that $\mathbb E[Z_n]= 1$, by criticality of $\mu$, and our standing hypothesis on $\varphi$ for the last inequality.

Now our recursion hypothesis takes the following form. 
\begin{lemma}\label{lem:boundsonkmoment}
\rm{	 There exists $C>0$, such that for any function $\varphi :  \mathbb Z^d\to [0,\infty)$, satisfying $\|G\varphi \|_\infty\le 1$, and any $x\in \mathbb Z^d$, we have for all $k\ge 1$, 
	\[
	\estart{\left(\sum_{v\in \cT_c} \phi(S_v) \right)^k}{x} \leq C^{k-1} \cdot ((k-2)\vee 1)! \cdot \estart{\sum_{v\in \cT_c}\phi(S_v)}{x}.
	\]}
\end{lemma}
Note that Theorem~\ref{theo.expmoment} for the critical tree $\cT_c$ immediately follows from this lemma and~\eqref{first.moment}. 
\begin{proof}[\bf Proof of Lemma~\ref{lem:boundsonkmoment}]

We will prove this by induction on $k$ and for a constant $C$ that we will determine. The case $k=1$ is immediate. We assume now that it holds for $k-1$ and we will prove it for $k$. We have 
\begin{align}\label{eq:sumofphiexpand}
	\estart{\left(\sum_{v\in \cT_c} \phi(S_v) \right)^k}{x} = \estart{\sum_{v_1,\ldots,v_k\in \cT_c} \phi(S_{v_1})\cdots \phi(S_{v_k})}{x}.
\end{align}
	We write $v_0={\rm{MRCA}}(v_1,\ldots,v_\ell)$, if $v_0$ is the most recent common ancestor of $v_1,\ldots, v_\ell$. With this notation we have 
	\begin{align}\label{eq:expressionforphi}
	\begin{split}
		&\sum_{v_1,\ldots,v_k\in \cT_c} \phi(S_{v_1})\cdots \phi(S_{v_k}) \\ &= \sum_{v_1,\ldots,v_k\in \cT_c}\sum_{v_0\in \cT_c}\1(v_0={\rm{MRCA}}(v_1,\ldots ,v_k))\1(v_0\notin \{v_1,\ldots, v_k\})\prod_{i=1}^{k} \phi(S_{v_i}) \\&+ \sum_{v_1,\ldots,v_k\in \cT_c}\sum_{v_0\in \cT_c}\1(v_0={\rm{MRCA}}(v_1,\ldots ,v_k))\1(v_0\in \{v_1,\ldots,v_k\})\prod_{i=1}^{k} \phi(S_{v_i})
	\end{split}
			\end{align}
	We first treat the case where $v_0\notin \{v_1,\ldots, v_k\}$. Changing the order of summation this becomes equal to 
	\begin{align*}
\sum_{v_0\in \cT_c} \sum_{j\geq 1}\1(\deg(v_0)=j)\sum_{y\in \Z^d}\1(S_{v_0}=y)\sum_{L=2}^{j}{j\choose L} \sum_{\substack{n_1,\ldots, n_L\\ \sum_{i}n_i=k\\ n_i\geq 1, \forall  i} }{k\choose n_1,\ldots, n_L}\prod_{\ell=1}^{L}\sum_{v_1^{\ell},\ldots, v_{n_\ell}^\ell\in \cT_\ell}\prod_{m=1}^{n_\ell}\phi(S_{v_m^{\ell}}),
	\end{align*} 
	where $\cT_\ell$ stands for the $\ell$-th descendant tree of $v_0$ containing at least one of the vertices $v_1,\dots,v_k$, and $n_\ell$ is the number of these vertices that it contains. The expectation of the expression above equals
	\begin{align*}
		\sum_{j\geq 1}\mu(j) \sum_{y\in \Z^d} g(x-y) \sum_{L=2}^{j} {j\choose L}\sum_{\substack{n_1,\ldots, n_L\\ \sum_{i}n_i=k\\ n_i\geq 1, 
		\forall  i} }{k\choose n_1,\ldots, n_L} \prod_{\ell=1}^{L}\estart{\left(\sum_{v\in \cT_c}\phi(S_v) \right)^{n_\ell}}{\nu_{y}},
	\end{align*} 
	where $\nu_y$ is the uniform distribution on all neighbours of $y$.
	We can now use the induction hypothesis to upper bound the expectations appearing above and obtain that this last expression  is bounded by 
	\begin{align*}
		\sum_{j\geq 1}\mu(j) \sum_{y\in \Z^d} g(x-y) \sum_{L=2}^{j} {j\choose L}\sum_{\substack{n_1,\ldots, n_L\\ \sum_{i}n_i=k\\ n_i\geq 1, 
		\forall  i} }{k\choose n_1,\ldots, n_L} \cdot C^{k-L}\cdot \prod_{i=1}^{L} ((n_i-2)\vee 1)! \cdot \left(\estart{\sum_{v\in \cT_c}\phi(S_v)}{\nu_y} \right)^{L}.
	\end{align*}
	By \eqref{first.moment} we get that the last expectation to the power $L$ can be bounded by its square, since we take $L\geq 2$. Expanding the combinatorial factor we deduce
	\begin{align}\label{eq:bigsum}
		\sum_{j\geq 1}\mu(j) \sum_{y\in \Z^d} g(x-y)\left(\estart{\sum_{v\in \cT_c}\phi(S_v)}{\nu_y}\right)^2 \cdot \sum_{L=2}^{j} {j\choose L}\cdot C^{k-L}\cdot \sum_{\substack{n_1,\ldots, n_L\\ \sum_{i}n_i=k\\ n_i\geq 1, \forall i} } k!\cdot \prod_{i=1}^{L} \frac{1}{n_i((n_i-1)\vee 1)}.
	\end{align}
	\begin{claim}\label{cl:induction}
		There exists a positive constant $C_1$ so that for every $L\geq 2$ we have 
		\[
		\sum_{\substack{n_1,\ldots, n_L\\ \sum_{i}n_i=k\\ n_i\geq 1, \forall  i} } \prod_{i=1}^{L} \frac{1}{n_i((n_i-1)\vee 1)} \leq \frac{(C_1)^L}{k^2}.
		\]
	\end{claim}
\begin{claim}\label{cl:boundongphisquare}
		There exists a positive constant $C_2$ so that for all $x$ we have 
		\[
		 \sum_{y\in \Z^d} g(x-y)\left(\estart{\sum_{v\in \cT_c}\phi(S_v)}{\nu_y}\right)^2\leq C_2 \cdot \estart{\sum_{v\in \cT_c}\phi(S_v)}{x}.
		\]
	\end{claim}
	We now complete the proof of the lemma and defer the proofs of the claims to the end of the section. 
Using the two claims above we deduce that the sum in~\eqref{eq:bigsum} is upper bounded by 
	\begin{align*}
		C_2\cdot  \estart{\sum_{v\in \cT_c}\phi(S_v)}{x}\cdot
\sum_{j\geq 1}\mu(j) \ \sum_{L=2}^{j} {j\choose L}\cdot C^{k-L}\cdot  k!\cdot \frac{C_1^L}{k^2}. 
	\end{align*}
Using that $k!/k^2\leq (k-2)!$ and ${j\choose L}\leq j^L/L!$ and taking $C>C_1^{100}$,  we obtain that the above sum is bounded by
		\begin{align*}
		C_2\cdot (k-2)!\cdot C^k\cdot \estart{\sum_{v\in \cT_c}\phi(S_v)}{x}\cdot  \sum_{L=2}^{k} \frac{C^{-99L/100}}{L!}\cdot \sum_{j\geq L} \mu(j)\cdot j^L\\
		\leq C_2\cdot (k-2)!\cdot C^k\cdot \estart{\sum_{v\in \cT_c}\phi(S_v)}{x}\cdot\sum_{L=2}^{k} C_3^L \cdot C^{-99L/100},
			\end{align*}
			where $C_3$ is a positive constant and where for this we used the assumption that $\mu$ has exponential moments.
	Taking further $C>C_3^{100}\vee 4$ we get that 
	the sum above is upper bounded by 
	\begin{align*}
		2C^{k-1.96} \cdot (k-2)! \cdot \estart{\sum_{v\in \cT_c}\phi(S_v)}{x}.
	\end{align*}
	So far we have established that
	\begin{align}\label{eq:boundwhenv0notoneofthem}
	\begin{split}
		\estart{\sum_{v_1,\ldots,v_k\in \cT_c}\sum_{v_0\in \cT_c}\1(v_0={\rm{MRCA}}(v_1,\ldots ,v_k))\1(v_0\notin \{v_1,\ldots, v_k\})\prod_{i=1}^{k} \phi(S_{v_i})}{x}\\ \leq 2C^{k-1.96} \cdot (k-2)! \cdot \estart{\sum_{v\in \cT_c}\phi(S_v)}{x}.
		\end{split}
	\end{align}
	We then have, denoting by $\cT_{v_0}$ the subtree of descendants of $v_0$, 
	\begin{align*}
				& \estart{\sum_{v_1,\ldots,v_k\in \cT_c}\sum_{v_0\in \cT_c} \1(v_0={\rm{MRCA}}(v_1,\ldots ,v_k) )\1(v_0\in \{v_1,\ldots, v_k\}) \prod_{i=1}^{k} \phi(S_{v_i})}{x} \\&=
				\estart{\sum_{v_0\in \cT_c} \sum_{\substack{v_1,\ldots, v_k\in \cT_{v_0}\\v_0\in \{v_1,\ldots,v_k\} }} \prod_{i=1}^{k}\phi(S_{v_i})}{x}\leq \estart{\sum_{v_0\in \cT_c}\sum_{y\in \Z^d}\1(S_{v_0}=y)\phi(y)\sum_{v_1,\ldots,v_{k-1}\in \cT_{v_0}}k \prod_{i=1}^{k-1}\phi(S_{v_i})  }{x} \\&= 
				\sum_{y\in \Z^d} \phi(y)g(x-y) k \cdot \estart{\left(\sum_{v\in \cT_c}\phi(S_v)\right)^{k-1}}{y}.
			\end{align*}
			We now consider two different cases depending on whether $k=2$ or $k\geq 3$. If $k=2$, then the sum above becomes equal to
			\[
			2 \sum_{y\in \Z^d} \phi(y)g(x-y)\estart{\sum_{v\in \cT_c}\phi(S_v)}{y} \leq 2\estart{\sum_{v\in \cT_c}\phi(S_v)}{x} \leq 2C^{k-1.96}\cdot \estart{\sum_{v\in \cT_c}\phi(S_v)}{x},
			\]
			since $C>4$ and where we used again \eqref{first.moment}.  
			Suppose next that $k\geq 3$. Then we can use the induction hypothesis to get
\begin{align*}
	\sum_{y\in \Z^d} \phi(y)g(x-y) k \cdot \estart{\left(\sum_{v\in \cT_c}\phi(S_v)\right)^{k-1}}{y}&\leq \sum_{y\in \Z^d} \phi(y) g(x-y) \cdot k \cdot (k-3)! C^{k-2} \cdot \estart{\sum_{v\in \cT_c}\phi(S_v)}{y}
	\end{align*}
Using that $k\leq 3(k-2)$ for $k\geq 3$ we can further bound the sum above by 
\[
 3(k-2)!\cdot C^{k-2}\cdot \estart{\sum_{v\in \cT_c}\phi(S_v)}{x}.
\]
Therefore in both cases taking $C>(3/2)^{25}$ we get 
\begin{align*}
	\estart{\sum_{v_1,\ldots,v_k\in \cT_c}\sum_{v_0\in \cT_c} \1(v_0={\rm{MRCA}}(v_1,\ldots ,v_k) )\1(v_0\in \{v_1,\ldots, v_k\}) \prod_{i=1}^{k} \phi(S_{v_i})}{x}\\ \leq 
	2C^{k-1.96}\cdot (k-2)!\cdot \estart{\sum_{v\in \cT_c}\phi(S_v)}{x}.
\end{align*}
Plugging in this bound together with the bound from~\eqref{eq:boundwhenv0notoneofthem}  into~\eqref{eq:expressionforphi} and using~\eqref{eq:sumofphiexpand} we finally deduce
\begin{align*}
	\estart{\left(\sum_{v\in \cT_c} \phi(S_v) \right)^k}{x} \leq 4C^{k-1.96}\cdot (k-2)!\cdot \estart{\sum_{v\in \cT_c}\phi(S_v)}{x}\leq C^{k-1} \cdot (k-2)!\cdot \estart{\sum_{v\in \cT_c}\phi(S_v)}{x},
	\end{align*}
since $C>(3/2)^{25}$ and this completes the proof.
\end{proof}

It remains to prove the two claims used in the proof above. 

\begin{proof}[\bf Proof of Claim~\ref{cl:induction}]
	We first note that there exists a positive constant $A>1$ such that for all $k$ we have 
	\begin{align}\label{eq:standardbound}
		\sum_{\ell=1}^{k} \frac{1}{\ell((\ell-1)\vee 1)}\cdot \frac{1}{(k-\ell)^2} \leq \frac{A}{k^2}. 
	\end{align}
We will prove that the statement of the claim is true for $C_1=A$ by induction on $L$. For $L=1$, the claim is obvious for all $k$. Suppose now that the claim is true for $L-1$ for all values of $k$. We will establish it also for $L$. We have 
\begin{align*}
	\sum_{\substack{n_1,\ldots, n_L\\ \sum_{i}n_i=k\\ n_i\geq 1, \forall i} } \prod_{i=1}^{L} \frac{1}{n_i((n_i-1)\vee 1)} &= \sum_{n_1=1}^{k} \frac{1}{n_1((n_1-1)\vee 1)} \cdot \sum_{\substack{n_2,\ldots, n_L\\ \sum_{i}n_i=k-n_1\\ n_i\geq 1, \forall  i}} \prod_{i=2}^{L} \frac{1}{n_i((n_i-1)\vee 1)} \\
&	\leq \sum_{n_1=1}^{k} \frac{1}{n_1((n_1-1)\vee 1)} \cdot \frac{(C_1)^{L-1}}{(k-n_1)^2} \leq C_1^{L-1}\cdot \frac{A}{k^2} = \frac{C_1^L}{k^2},
\end{align*}
	where for the first inequality we used the induction hypothesis and for second one we used~\eqref{eq:standardbound}. This completes the proof.
		\end{proof}

			\begin{proof}[\bf Proof of Claim~\ref{cl:boundongphisquare}]
	We first prove that there exists a universal constant $C$ so that for all $x,z,z'$ we have 
	\begin{align*}
		\sum_y g(x-y) g(y-z) g(y-z')\leq C (g(x-z) G(x-z') + g(x-z) G(z-z')).                      
	\end{align*}
	Indeed, letting $E=\{y: \|y-z\|\geq \|z-x\|/2\}$ we have for the sum over $E$
	\begin{align*}
		\sum_{y\in E}g(x-y) g(y-z) g(y-z')\lesssim \frac{1}{\|z-x\|^{d-2}} \cdot \sum_y g(x-y)  g(y-z') \lesssim g(x-z) G(x-z'), 
	\end{align*}
	using~\eqref{green.rw} and~\eqref{convol.G} for the last inequality.
	We notice that when $y\in E^c$, then by the triangle inequality we get $\|y-x\|\geq \|z-x\|/2$. So for the sum over $E^c$ we get
	\begin{align*}
		\sum_{y\in E^c}g(x-y) g(y-z) g(y-z')\lesssim \frac{1}{\|z-x\|^{d-2}} \cdot \sum_{y} g(y-z) g(y-z')\lesssim g(x-z)G(z-z').
	\end{align*}
We now prove the statement of the claim. We have, using the notation $x\sim y$ when $x$ and $y$ are neighbors in $\mathbb Z^d$,  
\begin{align*}
	 \sum_{y\in \Z^d} &g(x-y)\left(\frac{1}{2d}\sum_{z\sim y}\estart{\sum_{v\in \cT_c}\phi(S_v)}{z}\right)^2 \\
&	 = \frac{1}{(2d)^2}
	 \sum_{y\in \Z^d}\sum_{y_1\sim y}\sum_{y_2\sim y}\sum_{z,z'} g(x-y) g(y_1-z)g(y_2-z')\phi(z)\phi(z').
\end{align*}
Using that for all neighbours $y_1$ of $y$ and all $z$ we have that $g(y_1-z)\asymp g(y-z)$, we then get 
\begin{align*}
	 \sum_{y\in \Z^d} &g(x-y)\left(\frac{1}{2d}\sum_{z\sim y}\estart{\sum_{v\in \cT_c}\phi(S_v)}{z}\right)^2 
	 \asymp \sum_{y\in \Z^d}\sum_{z,z'} g(x-y) g(y-z)g(y-z')\phi(z)\phi(z') \\
	 &\lesssim \sum_{z,z'} (g(x-z) G(x-z') + g(x-z) G(z-z'))\phi(z)\phi(z')
	\leq 2\cdot \sum_{z}g(x-z) \phi(z), 
	 \end{align*}
where in the last inequality we used Claim~\ref{lem:boundonge}. Using that 
\[
\sum_{z}g(x-z) \phi(z)=\estart{\sum_{v\in \cT_c}\phi(S_v)}{x},
\]
finally concludes the proof of the claim.
\end{proof}


\subsection{Moments for the infinite tree}
Recall that we write $x\sim y$ when $x$ and $y$ are neighbours in $\mathbb Z^d$. We begin with the case of an adjoint branching random walk. 
\begin{lemma}
\rm{
	There exists $C>0$, so that for all $k\in \N$ and $z\in \Z^d$, we have 
	\[
	\estart{\left(\sum_{v\in \til{\cT}_c}\phi(S_v) \right)^k}{z}\leq C^{k-1} \cdot ((k-2)\vee 1)! \cdot \sup_{x\sim 0} \estart{\sum_{v\in {\cT}_c}\phi(S_v)}{z+x}.
	\]} 
\end{lemma}

\begin{proof}[\bf Proof]
We first claim that it suffices to prove that there exists $\lambda$ sufficiently small and a positive constant $C$ so that 
\begin{align}\label{eq:goalinthisproof}
	\estart{\Big(\sum_{v\in \til{\cT}_c}\phi(S_v) \Big)\cdot \exp\Big(\lambda\sum_{v\in \til{\cT}_c}\phi(S_v)  \Big)}{z} \leq C \cdot \sup_{x\sim 0} \estart{\sum_{v\in {\cT}_c}\phi(S_v)}{z+x}.
\end{align}
	Indeed, once this is established, then
	expanding the exponential we get 
	\begin{align*}
		\estart{\Big(\sum_{v\in \til{\cT}_c}\phi(S_v) \Big)\cdot \exp\Big(\lambda\sum_{v\in \til{\cT}_c}\phi(S_v)  \Big)}{z} = \sum_{n=0}^{\infty} \frac{\lambda^n}{n!}\cdot \estart{\Big(\sum_{v\in \til{\cT}_c}\phi(S_v) \Big)^{n+1}}{z},
	\end{align*}
	and hence for all $n\geq 1$ this gives 
	\[
	\estart{\Big(\sum_{v\in \til{\cT}_c}\phi(S_v) \Big)^{n}}{z} \leq C \cdot \frac{1}{\lambda^{n-1}} \cdot (n-1)! \cdot \sup_{x\sim 0} \estart{\sum_{v\in {\cT}_c}\phi(S_v)}{z+x},
	\]
	which is equivalent to the statement of the lemma by taking the constant $C$ from the statement sufficiently large. 
	
	We now turn to prove~\eqref{eq:goalinthisproof}. 
	Let $Z$ be the number of offspring of the root of $\til{\cT}_c$ which has distribution~$\til{\mu}$ and let $(U_i)_{i\ge 1}$ be i.i.d.\ uniformly chosen among the neighbours of $0$. Then we can write 
	\begin{align*}
		&\sum_{v\in \til{\cT}_c}\phi(S_v)\cdot \exp\Big(\lambda\sum_{v\in \til{\cT}_c}\phi(S_v)  \Big) = \sum_{i=1}^{Z}\Big(\sum_{v\in \cT_c^i} \phi(S_v^i)\Big)\cdot \exp\Big(\lambda \sum_{v\in \cT_c^i} \phi(S_v^i)\Big) \cdot \prod_{\substack{j\leq Z\\ j\neq i}}\exp\Big(\lambda \sum_{v\in \cT_c^j} \phi(S_v^j)\Big),
	\end{align*}
	where $(\cT_c^i)_{i\ge 1}$ are i.i.d.\ critical trees and $(S^i)_{i\ge 1}$ are independent branching random walks on~$(\cT^i_c)_{i\ge 1}$ started from $(U_i)_{i\ge 1}$. 
	Using the independence property, we then get 
\begin{align}\label{eq:firstandexp}
\begin{split}
	\estart{\Big(\sum_{v\in \til{\cT}_c}\phi(S_v) \Big)\cdot \exp\Big(\lambda\sum_{v\in \til{\cT}_c}\phi(S_v)  \Big)}{z} 
		= \sum_{k\in \N} &k\cdot \pr{Z=k}  \cdot \left(\estart{\exp\Big(\lambda \sum_{v\in \cT_c^1} \phi(S_v^1)\Big)}{z} \right)^{k-1} \\ 
		&\times \estart{\Big(\sum_{v\in \cT_c^1} \phi(S_v^1)\Big)\cdot \exp\Big(\lambda \sum_{v\in \cT_c^1} \phi(S_v^1)\Big)}{z} .
\end{split}
\end{align}
	By Lemma~\ref{lem:boundsonkmoment} we obtain for $\lambda<1/C$ with $C$ as in Lemma~\ref{lem:boundsonkmoment}
	\begin{align*}
		\estart{\Big(\sum_{v\in \cT_c^1} \phi(S_v^1)\Big)\cdot \exp\Big(\lambda \sum_{v\in \cT_c^1} \phi(S_v^1)\Big)}{z}  = \sum_{n=0}^{\infty} \frac{\lambda^n}{n!} \cdot \estart{\Big(\sum_{v\in \cT_c^1} \phi(S_v^1)\Big)^{n+1}}{z}  \\
		\leq \frac{1}{1-\lambda C}\cdot  \sup_{x\sim 0} \estart{\sum_{v\in \cT_c^1} \phi(S_v^1)}{z+x}.	\end{align*}
		We also get that for $\lambda<1/C$ 
		\begin{align*}
			\estart{\exp\Big(\lambda \sum_{v\in \cT_c^i} \phi(S_v^1)\Big)}{z} \leq \frac{1}{1-\lambda C} ,
		\end{align*}
		and hence plugging these two bounds into~\eqref{eq:firstandexp} yields 
		for $\lambda<1/C$
		\begin{align*}
			\estart{\Big(\sum_{v\in \til{\cT}_c}\phi(S_v) \Big)\cdot \exp\Big(\lambda\sum_{v\in \til{\cT}_c}\phi(S_v)  \Big)}{z} \leq \frac{1}{1-\lambda C}\cdot\sup_{x\sim 0} \estart{\sum_{v\in \cT_c^1} \phi(S_v^1)}{z+x}\cdot  \E{Z \left(\frac{1}{1-\lambda C} \right)^Z}.
		\end{align*}
		Since $\mu$ has an exponential moment, the same is true also for $\til{\mu}$. Therefore, choosing $\lambda$ sufficiently small we get that the last expectation appearing on the right hand side above is bounded by a constant, and hence this completes the proof of~\eqref{eq:goalinthisproof} and the proof of the lemma.
\end{proof}

We now move to the case of an infinite tree. We start with the case of $\cT_-$ which is slightly easier to handle.  
\begin{lemma}\label{lem.moment.infinite}
\rm{	There exists $C>0$ so that for all $z\in \Z^d$ and $n\geq 1$, we have 
	\[
	\estart{\left(\sum_{v\in \cT_-}\phi(S_v) \right)^n}{z}   \leq C^n \cdot  n!.
	\]}
	\end{lemma}
\begin{proof}[\bf Proof]
Let $(X_m)$ denote the random walk of the spine in its natural parametrization. We then have 
\begin{align*}
	\estart{\left(\sum_{v\in \cT_-}\phi(S_v) \right)^k}{z} = \estart{\left(\sum_{n=1}^{\infty}\sum_{u\in \til{\cT}_c^n}\phi(S_u^n+X_n) \right)^k}{z},
\end{align*}
	where $(\til{\cT}_c^n)$ are i.i.d.\ adjoint critical trees and $(S^n)$ are independent branching random walks on~$(\til{\cT}_c^n)$. With this representation we now obtain
	\begin{align}\label{eq:verylongeq}
			\nonumber &\estart{\left(\sum_{v\in \cT_-}\phi(S_v) \right)^n}{z}  
		= \sum_{j=1}^{n}\sum_{\substack{n_1,\ldots, n_j\geq 1\\ \sum_{i=1}^{j}n_i=n}} {n\choose n_1,\ldots, n_j} \sum_{k_1<\ldots < k_j} \estart{\prod_{i=1}^{j} \left(\sum_{u\in \til{\cT}_c^{k_i}} \phi(S_u^{k_i}+X_{k_i})\right)^{n_i}}{z} 
	\\&=\sum_{j=1}^{n}\sum_{\substack{n_1,\ldots, n_d\geq 1\\ \sum_{i=1}^{j}n_i=n}} {n\choose n_1,\ldots, n_j}  \sum_{x_1,\ldots, x_j} g(z,x_1)\ldots g(x_{j-1},x_j) \prod_{i=1}^{j} \estart{\left(\sum_{u\in \til{\cT}_c} \phi(S_u)\right)^{n_i}}{x_i}\\
		\nonumber &=\sum_{j=1}^{n}\sum_{\substack{n_1,\ldots, n_j\geq 1\\ \sum_{i=1}^{j}n_i=n}} {n\choose n_1,\ldots, n_j}  \sum_{x_1,\ldots, x_j} g(z,x_1)\ldots g(x_{j-1},x_j) 
		\cdot \prod_{i=1}^{j} C^{n_i-1} (n_i-2)! \sup_{y\sim 0}\estart{\sum_{u\in {\cT}_c} \phi(S_u)}{x_i+y},
	\end{align}
	where for the last step we used Lemma~\ref{lem:boundsonkmoment}. Using Claim~\ref{cl:induction} the above sum reduces to 
	\begin{align}\label{eq:sumusingclaim}
	\sum_{j=1}^{n} n!\cdot \frac{C_1^j}{n^2}\cdot C^{n-j} \sum_{x_1,\ldots, x_j} g(z,x_1)\ldots g(x_{j-1},x_j) \cdot \prod_{i=1}^{j}  \sup_{y\sim 0}\estart{\sum_{u\in {\cT}_c} \phi(S_u)}{x_i+y}.
	\end{align} 
	Furthermore, for all $x,z\in \Z^d$ and $y\sim 0$ we have $g(z,y)\asymp g(z,x+y)$ and by~\eqref{convol.G}, 
	\[
	\sum_{x}g(z,x) \estart{\sum_{v\in \cT_c}\phi(S_v)}{x} = \sum_{x,y}g(z,x)g(x,y)  \varphi(y) \lesssim G\varphi(z) \lesssim 1. 
	\]
	These now imply that there exists a positive constant~$C_2$ so that 
	\[
	\sum_{x_1,\ldots, x_j} g(z,x_1)\ldots g(x_{j-1},x_j) \cdot \prod_{i=1}^{j}  \sup_{y\sim 0}\estart{\sum_{u\in {\cT}_c} \phi(S_u)}{x_i+y} \leq C_2^j.
	\]
	Plugging this back into~\eqref{eq:sumusingclaim} and then into~\eqref{eq:verylongeq} we conclude that 
	\begin{align*}
	\estart{\left(\sum_{v\in \cT_-}\phi(S_v) \right)^n}{z}  
	\leq (n-2)!\cdot \sum_{j=1}^{n} C^{n-d} \cdot C_2^j \cdot C_1^j,
	\end{align*}
	which by choosing $C$ sufficiently large compared to $C_1$ and $C_2$ finishes the proof.
\end{proof}

\begin{proof}[\bf Proof of Theorem~\ref{theo.expmoment}]
Lemmas~\ref{lem:boundsonkmoment} and~\ref{lem.moment.infinite} immediately imply that the statement of the theorem is true for $\cT_c$ and $\cT_-$. 
Moreover, $\cT_+$ is made of a critical tree attached to the root plus a forest of trees which are contained in a copy of $\cT_-$. Therefore the result for $\cT$ follows by an application of the Cauchy-Schwarz inequality.  
\end{proof}


\subsection{Proof of Corollary~\ref{cor.loctimeball}}
We start with two technical lemmas. Recall that for $\mathcal C\subset \mathbb Z^d$, we write $B(\mathcal C,r) = \cup_{x\in \mathcal C}B(x,r)$. 

\begin{lemma}\label{lem:boundonge}
\rm{	There exists $C_1>0$, so that the following holds. Let $\cC$ be a finite collection of points in $\Z^d$ within distance greater than $2r$ from each other. Then for all $x_0\in \Z^d$ we have 
	\begin{equation*}
	\sum_{x\in \cC}\sum_{y\in B(x,r)} G(x_0-y) \sum_{z\in \partial B(x,r)}e_{B(\cC,r)}(z)\leq C_1 \cdot r^d. 
	\end{equation*}
	}
\end{lemma}

\begin{proof}[\bf Proof]
We let 
\[
A(x_0) = \{ x\in \cC: \norm{x-x_0}\geq 2r\}.
\]
By~\eqref{Green.asymp}, there exists $C_1>0$, such that for all $x\in A(x_0)$, all $y\in B(x,r)$ and $z\in \partial B(x,r)$,
\[
G(x_0-y) \leq C_1 \cdot G(x_0 -z).
\]
We then have 
\begin{align*}
	\sum_{x\in A(x_0)}\sum_{y\in B(x,r)} G(x_0-y) \sum_{z\in \partial B(x,r)}e_{B(\cC,r)}(z) &\leq C_1C_2 r^d\sum_{x\in A(x_0)}\sum_{z\in \partial B(x,r)} G(x_0-z)e_{B(\cC,r)}(z)\\ &\leq C_1C_2 r^d \sum_{x\in \cC}\sum_{z\in \partial B(x,r)} G(x_0-z)e_{B(\cC,r)}(z)\leq C_4 r^d,
\end{align*}
where for the last step we used~\eqref{Upper.theo}. 
For the sum over $x\notin A(x_0)$, we have using again~\eqref{Green.asymp}, 
\begin{align*}
	\sum_{x\notin A(x_0)}\sum_{y\in B(x,r)} G(x_0-y) \sum_{z\in \partial B(x,r)}e_{B(\cC,r)}(z) &\leq \cpc{B(x,r)}\cdot  \sum_{x\notin A(x_0)}\sum_{y\in B(x,r)} G(x_0-y) \\
	&\lesssim r^{d-4} \sum_{x\in B(x_0,3r)} G(x_0-x) \asymp r^{d-4} \cdot r^4 = r^{d}.
	\end{align*}
	This now concludes the proof.
\end{proof}

\begin{lemma}\label{lem.double}
\rm{ There exists $c_1>0$, such that for any finite set $\cC\subset \mathbb Z^d$, and any $r\ge 1$, one has 
$$\cpc{B(\cC,r)} \ge c_1\cdot \cpc{B(\cC,3r)}. $$ }
\end{lemma}
\begin{proof}[\bf Proof]
Note that $B(\cC,3r)$ can be covered by a finite number of translates of $B(\cC,r)$, so the lemma just follows from sub-additivity of branching capacity proved in~\cite{Zhu1}. 
\end{proof}

\begin{proof}[\bf Proof of Corollary~\ref{cor.loctimeball}] 
Let $\cC$ be a finite subset of $\mathbb Z^d$ and $r\ge 1$ be given. 
By discarding some points, one can find a subset $\cC'\subseteq \cC$, whose points are all at distance greater than $2r$ from each other, and such that $B(\cC,r) \subseteq B(\cC',3r)$. 

We now define a function $\phi$ by taking it to be equal to $0$ outside of $B(\cC',r)$ and for every $y\in B(x,r)$ with $x\in \cC'$ we define
\begin{align}\label{eq:defofphi}
\phi(y)= \frac{1}{C_1r^d} \cdot \sum_{z\in \partial B(x,r)} e_{B(\cC',r)}(z),
\end{align}
where $C_1$ is the constant from Lemma~\ref{lem:boundonge} so that $\|G\varphi \|_\infty \le 1$. 
We then have using Chernoff's bound
\begin{align*}
	\pr{\ell_{\cT^0}(B(x,r))\geq t, \ \forall \ x\in \cC} & \leq \pr{\ell_{\cT^0}(B(x,r))\geq t, \ \forall \ x\in \cC'} \\
	& \le \pr{\sum_{x\in \cC'}\phi(x) \cdot \ell_{\cT^0}(B(x,r)) \geq t \cdot \frac{\cpc{B(\cC',r)}}{C_1r^d} }\\
	& \le 2\exp\Big(-\kappa \cdot t \cdot \frac{\cpc{B(\cC',r)}}{C_1r^d}\Big)\\
	&\le 2 \exp\Big(-\kappa c_1\cdot t \cdot \frac{\cpc{B(\cC,r)}}{C_1r^d}\Big), 
\end{align*}
where we used Theorem~\ref{theo.expmoment} at the third line, and Lemmas~\ref{lem.double} and~\ref{lem-shrink} at the last one.  
This concludes the proof of the corollary. 
\end{proof}

\begin{remark}\label{rem.balls.rw}
\rm{In the case of a simple random walk on $\mathbb Z^d$, $d\ge 3$, one can recover Theorem 1.2 of~\cite{AS23a} by using a similar argument (which in the setting of standard random walks is much simpler). This allows also to remove the hypothesis (1.4) from there. 
}
\end{remark}


\section{The Lower Bound in Theorem~\ref{theo.potentiel}} \label{sec.Lowerbound}

\subsection{Preliminary estimates}

	Given $A$ and $B$ two disjoint subsets of $\mathbb Z^d$, we say 
that a tree-indexed random walk hits the set~$A$ before the set $B$, if the first vertex in the lexicographical order of the tree at which the walk is in~$A\cup B$, the walk is in $A$. We say that the tree indexed random walk hits the set $A$ after the set~$B$, if it hits the set $A$ but not before the set $B$.

\begin{lemma}\label{lem.pre.1}
\rm{
There exist positive constants $c$ and $L_0\ge 3$, such that for any $L\ge L_0$, any $R\ge 1$, any finite set $K\subseteq B(0,R)$, and any $x\in  B(0,LR)\smallsetminus B(0,2R)$, 
$$\mathbb P\Big(\cT^x_c \text{ hits }K \text{ before } \partial B(0,L^2R) \Big)\ge c \cdot \frac{\cpc{K}}{(LR)^{d-2}}. $$ 
}
\end{lemma}
\begin{proof}[\bf Proof]
For two vertices $u,v\in \mathcal T_c$, let us write $v<u$ if $v$ is on the geodesic going from the root to $u$, and different from $u$. Then consider the set 
$$\mathcal U^x = \{u\in \cT_c : S_u^x \in \partial B(0,L^2R) \text{ and } S_v^x \in B(0,L^2R)\setminus \partial B(0,L^2R) \ \forall v<u\}. $$ 
If $\cT^x_c$ hits $K$, but only after hitting $\partial B(0,L^2R)$, there must exist $u\in \mathcal U^x $, whose tree of descendants hits $K$. 
Since conditionally on $\mathcal U^x $, the descendant trees of its vertices are independent copies of $\cT_c$, we get 
\begin{align}\label{eq.hit.before}
 \mathbb P\Big(\cT^x_c \text{ hits }K \text{ after } \partial B(0,L^2R) \mid \mathcal U^x  \Big)&\le |\mathcal U^x|\cdot \sup_{x\in \partial B(0,L^2R)} \mathbb P(\cT^x_c \cap K \neq \emptyset)  
& \lesssim |\mathcal U^x|\cdot \frac{\cpc{K}}{(L^2R)^{d-2}}, 
\end{align} 
using~\eqref{hitting-critical} for the last inequality. Note also that for any $x\in B(0,LR)$, 
$$\mathbb E[|\mathcal U^x|] =  \sum_{n\ge 0} \mathbb E[Z_n]\cdot \mathbb P_x(\tau = n) = 1,$$
where $Z_n$ is the number of vertices of generation $n$ in $\cT_c$, and $\tau$ is the hitting time of $\partial B(0,L^2R)$ by a simple random walk. 
Therefore taking expectation on both sides of~\eqref{eq.hit.before} gives
$$\mathbb P\Big(\cT^x_c \text{ hits }K \text{ after } \partial B(0,L^2R)  \Big) \lesssim \frac{\cpc{K}}{(L^2R)^{d-2}}. $$ 
Hence, applying again~\eqref{hitting-critical} gives for $L$ large enough, 
\begin{align*}
 \mathbb P\Big(\cT^x_c \text{ hits }K \text{ before } \partial B(0,L^2R) \Big) & = \mathbb P(\cT^x_c\cap K\neq \emptyset) - \mathbb P\Big(\cT^x_c \text{ hits }K \text{ after } \partial B(0,L^2R) \Big) \\
 & \gtrsim \frac{\cpc{K}}{(LR)^{d-2}}, 
 \end{align*} 
 concluding the proof of the lemma.
\end{proof}

Recall the definitions of $s(\gamma)$ and $b_A(x)$ given in Section~\ref{sec.hit}. 
For a path $\gamma$ define 
$$b_A(\gamma) = \prod_{\ell = 1}^{|\gamma|} b_A(\gamma(\ell)), \quad\text{and}\quad \til b_A(\gamma) = \prod_{\ell = 0}^{|\gamma|-1} b_A(\gamma(\ell)).$$

\begin{lemma}\label{lem.pre.2}
\rm{There exists $L_0>0$, such that for any $L\ge L_0$, and any $R\ge 1$, 
$$
\sum_{x\in \partial B(0,R)} \ \sum_{\substack{\gamma : x\to \partial B(0,LR)\\ R^2\le |\gamma|\le L^dR^2}} s(\gamma) \cdot b_{B(0,R)}(\gamma) 
  \ge \frac {\cpc{B(0,R)}}2. 
$$
}
\end{lemma}
\begin{proof}[\bf Proof]
For $x\in \partial B(0,R)$, consider a random walk indexed by $\cT$ starting from $x$, and let $\sigma^x$ be the first hitting time of $\partial B(0,LR)$ by the spine (using its intrinsic labelling). Then one has by definition, 
$$ \sum_{\substack{\gamma : x\to \partial B(0,LR)\\ R^2\le |\gamma|\le L^dR^2}} s(\gamma) \cdot b_{B(0,R)}(\gamma) \ge  
\frac{\mathbb P(\cT_-^x\cap B(0,R)=\emptyset, \, R^2\le \sigma^x \le L^dR^2)}{\sup_{y\in \partial B(0,LR)} \mathbb P(\cT^y_-\cap B(0,R) =\emptyset)}, $$ 
and thus by~\eqref{hitting-infinite} and Proposition~\ref{lem.bcapballs} one has for $L$ large enough, 
\begin{equation}\label{eq1.lem.pre.2} 
\sum_{\substack{\gamma : x\to \partial B(0,LR)\\ R^2\le |\gamma|\le L^dR^2}} s(\gamma) \cdot b_{B(0,R)}(\gamma) \ge  
\frac 34\cdot \mathbb P(\cT_-^x\cap B(0,R)=\emptyset, \, R^2\le \sigma^x \le L^dR^2). 
\end{equation} 
Recall that typically $\sigma^x$ is of order $L^2R^2$, hence 
one can expect the two events $\{\sigma^x> L^dR^2\}$ and $\{\sigma^x < R^2\}$ to have small probability, provided that $L$ is large enough. 
We start considering the first one. Let $\tau^x$ be the last visiting time of $\partial B(0,LR)$ by the walk on the spine: 
$$\tau^x  =\sup \{n\ge 0 : X^x(n) \in B(0,LR) \}. $$  
Then by rerooting the tree at the vertex corresponding to $\tau^x$, 
we can write using Proposition~\ref{prop.hit.Zhu}, and denoting by $Z_n$ the number of vertices at generation $n$ in a critical tree,   
\begin{align*} 
& \sum_{x\in \partial B(0,R)}  \mathbb P(\cT_-^x\cap B(0,R)=\emptyset, \, \sigma^x > L^dR^2)  \le  \sum_{x\in \partial B(0,R)} \mathbb P(\cT_-^x\cap B(0,R)=\emptyset, \,  \tau^x > L^dR^2)\\
&= \sum_{x\in \partial B(0,R)}\sum_{y\in \partial B(0,LR)}\mathbb P(\cT_-^x\cap B(0,R)=\emptyset, \,  \tau^x > L^dR^2, X^x(\tau^x) = y) \\
& = \sum_{y\in \partial B(0,LR)} \Big( \sum_{\substack{\gamma : y \to B(0,R) \\ |\gamma|> L^dR^2}} s(\gamma)\cdot \til b_{B(0,R)}(\gamma)\Big) \cdot \mathbb P(\cT^y_- \cap B(0,R) = \emptyset, X^y(n)\in B(0,LR)^c, \, \text{for all }n\ge 1) \\
 & \le \sum_{y\in \partial B(0,LR)} \mathbb P(Z_{L^dR^2} \neq 0) \cdot \mathbb P_y(X_n\in B(0,LR)^c, \, \text{for all }n\ge 1) \\
 & \lesssim \frac{\text{Cap}(B(0,LR))}{L^dR^2}\lesssim \frac{\cpc{B(0,R)}}{L^2},
 \end{align*}
 using also Kolmogorov's estimate at the last line, see e.g.~\cite[Theorem 1, p.19]{AN}. 
  Therefore, for~$L$ large enough, one has for any $R\ge 1$, 
  \begin{equation}\label{eq2.lem.pre.2}
 \sum_{x\in \partial B(0,R)}  \mathbb P(\cT_-^x\cap B(0,R)=\emptyset, \, \sigma^x > L^dR^2)  \le \frac {\cpc{B(0,R)}}{10}. 
   \end{equation}
   It remains to consider the event $\{\sigma^x< R^2\}$, which is more complicated to handle. We introduce two intermediate surfaces: 
   $$\Sigma_1 =\partial B(0,\frac {R\sqrt L}2), \quad \text{and}\quad \Sigma_2 = \partial B(0,R\sqrt L). $$  
   Define $\tau_1^x$ and $\tau_2^x$ to be the last visiting times of $\Sigma_1$ and $\Sigma_2$ respectively by the spine (for its natural parametrisation). 
   First observe that 
     \begin{align*}
     \mathbb P(\cT^x_- \cap B(0,R) = \emptyset, \, \tau_2^x>\sigma^x)
     & \leq \pr{\tau_2^x>\sigma^x} \pr{\cF_-^x[0,\sigma^x]\cap B(0,R)=\emptyset} \\
   & \lesssim \frac{1}{L^{\frac{d-2}{2}}}\cdot  \pr{\cF_-^x[0,\sigma^x]\cap B(0,R)=\emptyset}.
    \end{align*} 
  Using~\eqref{hitting-infinite} we see that for $L$ sufficiently large
  \begin{align*}
  	\pr{\cT^x_- \cap B(0,R)=\emptyset} &\geq  \pr{\cF_-^x[0,\sigma^x]\cap B(0,R)=\emptyset} \cdot \inf_{y\in \partial B(0,RL)} \pr{\cT_-^y\cap B(0,R)=\emptyset} \\ 
  	&\gtrsim \pr{\cF_-^x[0,\sigma^x]\cap B(0,R)=\emptyset},
  \end{align*}  
         and hence plugging this above we deduce
         \begin{align}\label{eq3.lem.pre.2}
         	\mathbb P(\cT^x_- \cap B(0,R) = \emptyset, \, \tau_2^x>\sigma^x) \lesssim \frac{1}{L^{\frac{d-2}{2}}}\cdot \pr{\cT^x_- \cap B(0,R)=\emptyset}. 
         \end{align}
    Now, denoting by $H_{\Sigma_1}$ and $H^+_{\Sigma_2}$ for the first hitting time of $\Sigma_1$ and first return time to $\Sigma_2$ respectively, by a simple random walk, one can write for some constant $c>0$, 
 \begin{align}\label{eq4.lem.pre.2}
\nonumber &  \sum_{x\in \partial B(0,R)} \mathbb P(\cT^x_- \cap B(0,R) = \emptyset, \, \tau_2^x<R^2) \\
\nonumber  & = \sum_{x\in \partial B(0,R)} \sum_{y\in \Sigma_1, z\in \Sigma_2} \mathbb P(\cT^x_- \cap B(0,R) = \emptyset, \, \tau_2^x<R^2, \, X^x(\tau_1^x)=y, X^x(\tau_2^x) = z) \\
\nonumber  & {\le} \sum_{y\in \Sigma_1, z\in \Sigma_2} \mathbb P(\cT^y_c \cap B(0,R)\neq \emptyset) \cdot \Big(\sum_{\substack{\gamma : y\to z\\ \gamma \subseteq \Sigma_1^c,\, |\gamma|< R^2}} s(\gamma)\Big) \cdot \mathbb P_z(H_{\Sigma_2}^+= \infty),
\end{align}
where the last inequality follows from Proposition~\ref{prop.hit.Zhu}. Using~\eqref{hitting-critical} for a positive constant $C$ we can now upper bound this last expression by 
\begin{align}
\nonumber  & \frac{C}{L^{\frac{d-2}{2}}  R^2} \cdot \sum_{z\in \Sigma_2} \mathbb P_z (H_{\Sigma_1} < R^2) \cdot \mathbb P_z(H_{\Sigma_2}^+= \infty)\\
  & \lesssim \exp(-c\sqrt L)  \cdot R^{d-4} {\lesssim} \exp(-c\sqrt L) \cdot  \cpc{B(0,R)},
 \end{align}   
 where for the last inequality we used~{\eqref{lem.bcapballs}}.
   Combining~\eqref{eq3.lem.pre.2} and~\eqref{eq4.lem.pre.2} yields for $L$ large enough, 
$$
  \sum_{x\in \partial B(0,R)} \mathbb P(\cT^x_- \cap B(0,R) = \emptyset, \, \sigma^x<R^2) \le \frac{\cpc{B(0,R)} }{10}.
$$
Together with~\eqref{eq2.lem.pre.2}, this gives 
$$ \sum_{x\in \partial B(0,R)} \mathbb P(\cT^x_- \cap B(0,R) = \emptyset, \, R^2 \le \sigma^x \le L^dR^2) \ge \frac 45\cdot \cpc{B(0,R)},$$ 
and remembering also~\eqref{eq1.lem.pre.2} concludes the proof of the lemma. 
\end{proof}

\subsection{Proof of the lower bound of Theorem~\ref{theo.potentiel}}

\begin{proof}[\bf Proof of~\eqref{Lower.theo}]
Assume without loss of generality that $0\in K$, and $x=0$. It amounts to bound from below $Ge_K(0)=\sum_{x\in K} G(x)e_K(x)$, 
by some universal constant that does not depend on $K$. 
Fix $L\ge 2$ to be determined later and define $R_i = L^{2i}$, for $i\ge 0$. Then let 
$$B_i = B(0,R_i), \quad \mathcal S_i = B_i\smallsetminus B_{i-1} \quad \text{and}\quad K_i = K\cap B_i,$$ 
with also $B_{-1}= \emptyset$.  
Define $I$ as the maximal index $i$ such that $K\cap \mathcal S_i \neq \emptyset$. 
Note that if $I\le 1$, then we can write 
$$Ge_K(0) \ge \big(\inf_{x\in B(0,R_1)} G(x)\big) \cdot \cpc{K} \ge \big(\inf_{x\in B(0,R_1)} G(x)\big) \cdot \cpc{\{0\}},$$
which gives a universal lower bound independent of $K$. Thus we may assume now that $I\ge 2$.

Recall that we defined $G(r) = r^{4-d}$, for $r>0$. Using~\eqref{Green.asymp}, we get that for a positive constance $c_0$ (only depending on $L$) whose value may change from line to line
\begin{equation}\label{shell-1}
\begin{split}
\sum_{x\in K} G(x)e_K(x)& = \sum_{i=0}^I\sum_{x\in \mathcal S_i} G(x)e_K(x)\ge {c_0} \sum_{i=0}^I G(R_i)\cdot e_K(\mathcal S_i)\\
& \ge  c_0
 \sum_{i=0}^I \big(\sum_{j\ge i} G(R_j)\big)\cdot  e_K(\mathcal S_i)=  c_0 \sum_{i=0}^I G(R_i)\cdot e_K(B_i).
\end{split}
\end{equation}
For $i\in \{0,\dots,I\}$, define 
$$\varepsilon_i = G(R_i)\cdot \cpc{K_i},$$ 
and let 
$$I^*= \inf\big\{i \ge 0 : \sum_{k=i}^I \varepsilon_k \le \delta\big\},$$
where $\delta>0$ is another constant to be fixed later, and using the convention $\inf \emptyset =+\infty$. 
The proof of~\eqref{Lower.theo} will follow from the next result, where we use the convention $K_{-1} = \emptyset$. 
\begin{proposition}\label{prop-invisible}
\rm{There exists $c>0$, and a choice of $L$ and $\delta$, such that for any finite $K\subset \mathbb Z^d$, 
and any index $i$ satisfying $I^*+1\le i \le I$, 
\begin{equation*}
e_K(B_i) \ge c \cdot \cpc{K_{i-2}}.  
\end{equation*}
}
\end{proposition}
Assuming this proposition, one can conclude the proof of~\eqref{Lower.theo}. Indeed, fix $L$ and $\delta$, as in Proposition~\ref{prop-invisible}, 
and distinguish between a few cases.
If $ \varepsilon_I \ge \delta/(4L^{2(d-4)})$, then we have by~\eqref{shell-1},  
$$Ge_K(0) \ge  c_0  G(R_I) \cdot e_K(B_I)  =c_0\cdot \varepsilon_I \ge c_0\cdot \frac{\delta}{4L^{2(d-4)}}. $$ 
If $\varepsilon_I\le \delta/(4L^{2(d-4)})$, then we have 
$$\varepsilon_{I-1} = G(R_{I-1})\cdot \cpc{K_{I-1}} \le L^{2(d-4)} \cdot \varepsilon_I \le \delta/4. $$  
In particular $I^*\le I-1$. If in addition $ I^*\ge 1$, then by~\eqref{shell-1} and Proposition~\ref{prop-invisible}, we get
$$Ge_K(0) \geq c_0 \sum_{i=I^*+1}^I G(R_i)\cdot e_K(B_i)\ge \frac{c_0c}{L^{4(d-4)}} \cdot \sum_{i=I^*-1}^{I-2} \varepsilon_i\ge \frac{c_0c}{L^{4(d-4)}} \cdot (\delta - \frac{\delta}{4} - \frac{\delta}{4L^{2(d-4)}}) \ge \frac{c_0c\cdot \delta}{2L^{4(d-4)}} . $$ 
If $I^*=0$, then we have as well (recall that we assume $I\ge 2$),
$$Ge_K(0) \ge c_0G(R_2) \cdot e_K(B_2) \ge c_0c G(R_2) \cdot \cpc{K_0} \ge c_0c G(R_2)\cdot \cpc{\{0\}}, $$
using that $K_0$ contains the origin for the last inequality.  
In all cases we get a universal lower bound for $Ge_K(0)$, independent of $K$, and this concludes the proof of~\eqref{Lower.theo}. 
\end{proof}

It remains to prove the previous proposition. 

\begin{proof}[\bf Proof of Proposition~\ref{prop-invisible}] 
Assume that $I^*\le I-1$, as otherwise there is nothing to prove, and fix some $i\in \{I^*+1,\dots,I\}$. 
By Lemma~\ref{lem-shrink}, we have that 
\begin{equation}\label{shell-2}
e_K(B_i)\ge e_{K^i} (B_{i-2}),
\quad\text{where}\quad K^i=K\setminus (\mathcal S_i\cup \mathcal S_{i-1}).
\end{equation}
Applying Proposition~\ref{pro:lastpassagedec}  yields 
\begin{equation}\label{shell-3}
e_{K^i}(B_{i-2})  = \sum_{w\in \partial B_{i-1}} \mathbb P\big(\cT^w_+ \text{ first hits }K^i \text{ in } K_{i-2}, \, \cT^w_- \cap (B_{i-1}\cup K^i) =\emptyset\big). 
\end{equation}
Let $\Sigma = \partial B(0,L R_{i-1})$, and define $\sigma$ as the first time the spine hits $\Sigma$ (in its natural parametrisation). 
One has for any $w\in \partial B_{i-1}$, 
\begin{align}\label{lower.prop.proof}
 \nonumber  \mathbb P\big(\cT^w_+  & \text{ first hits }K^i \text{ in } K_{i-2}, \, \cT^w_- \cap (B_{i-1}\cup K^i) =\emptyset\big) \\
& \ge \nonumber
\mathbb P(\cF_+^w[0,\sigma] \text{ first hits }K^i \text{ in } K_{i-2}, \, \cT^w_- \cap (B_{i-1}\cup K^i)=\emptyset\big)  \\
\nonumber & \ge \mathbb P(\cF_+^w[0,\sigma] \text{ hits } K_{i-2} \text{ before } \partial B_i, \, \cF^w_-[0,\sigma]  \cap (B_{i-1}\cup \partial B_i)=\emptyset\big) \\
&  \qquad \times \inf_{u \in \Sigma} \mathbb P\big(\cT^u_- \cap (B_{i-1}\cup K^i)=\emptyset \big).
\end{align}
We deal first with the last probability. 
By Proposition~\ref{lem.bcapballs} and~\eqref{hitting-infinite}, one has 
$$\sup_{u\in \Sigma} \mathbb P(\cT^u_- \cap B_{i-1} \neq \emptyset) \lesssim \frac{1}{L^{d-4}} ,$$
and by definition of $I^*$, one has by a union bound
$$\sup_{u\in \Sigma} \mathbb P(\cT^u_- \cap K \cap B_i^c\neq  \emptyset )  \lesssim \sum_{j\ge i +1}\varepsilon_j \lesssim \delta, $$ 
so that by choosing $\delta$ small enough, and $L$ large enough, one can ensure that 
\begin{equation}\label{inf.sigma}
\inf_{u\in \Sigma}  \mathbb P\big(\cT^u_- \cap (B_{i-1}\cup K^i)=\emptyset \big) \ge \frac 12. 
\end{equation}
Now we bound from below the other probability in~\eqref{lower.prop.proof}. Recall that $X^w$ denotes the random walk on the spine. 
One has 
\begin{align*}
&\mathbb P(\cF_+^w[0,\sigma] \text{ hits } K_{i-2} \text{ before } \partial B_i, \, \cF^w_-[0,\sigma]  \cap (B_{i-1}\cup \partial B_i)=\emptyset\big)  \\
& \ge \sum_{\substack{\gamma : w\to \Sigma \\ \gamma \subseteq B_{i-1}^c\\  R_{i-1}^2 \le |\gamma|\le L^dR_{i-1}^2}} s(\gamma) \cdot  
\mathbb P(\cF_+^w[0,\sigma] \text{ hits } K_{i-2} \text{ before } \partial B_i, \, \cF^w_-[0,\sigma]  \cap (B_{i-1}\cup \partial B_i)=\emptyset \mid X^w[0,\sigma]=\gamma\big)
\end{align*}
For $\ell \ge 0$, and on the event $\{X^w[0,\sigma] = \gamma\}$, we denote by $\til{\cT}^{\gamma(\ell)}_-$ the adjoint tree hanging off $\gamma(\ell)$ in the past and by $\til{\cT}^{\gamma(\ell)}_+$ the adjoint tree without its root hanging off $\gamma(\ell)$ in the future. 
Then using the independence of these trees for different $\ell$, we get that for any $\gamma : w\to \Sigma$ with $\gamma \subseteq B_{i-1}^c$, one has 
\begin{align}\label{eq.forest.1} 
\nonumber &\mathbb P(\cF_+^w[0,\sigma]\text{ hits } K_{i-2} \text{ before } \partial B_i, \, \cF^w_-[0,\sigma]  \cap (B_{i-1}\cup \partial B_i)=\emptyset \mid X^w[0,\sigma]=\gamma\big)\\
\nonumber & \ge \Big(\prod_{\ell=0}^{|\gamma|} \mathbb P(\til{\cT}^{\gamma(\ell)}_+ \cap (B_{i-2}\cup \partial B_i)=\emptyset,\, \til{\cT}^{\gamma(\ell)}_- \cap (B_{i-1}\cup \partial B_i)=\emptyset)  \Big)
\\
& \qquad \times \Big( \sum_{\ell=0}^{|\gamma|} \mathbb P (\til{\cT}^{\gamma(\ell)}_+ \text{ hits } K_{i-2} \text{ before } \partial B_i,\, \til{\cT}^{\gamma(\ell)}_-\cap (B_{i-1}\cup \partial B_i)=\emptyset)  \Big). 
\end{align}
Now by~\eqref{hitting-critical} and since $\gamma\subseteq B_{i-1}^c\cap B(0,LR_{i-1})$, one has for any $\ell \ge 0$, and some constant $c>0$, 
\begin{align*}
\mathbb P(\til{\cT}^{\gamma(\ell)}_+ \cap (B_{i-2}\cup \partial B_i)=\emptyset, \, \til{\cT}^{\gamma(\ell)}_- \cap (B_{i-1}\cup \partial B_i)=\emptyset) 
& \ge \mathbb P(\til{\cT}^{\gamma(\ell)}_- \cap B_{i-1}=\emptyset) \cdot (1- \frac{c}{R_{i-1}^2}) \\
&=b_{B_{i-1}}(\gamma(\ell)) \cdot (1- \frac{c}{R_{i-1}^2}),  
\end{align*}
and thus the product on the right-hand side of~\eqref{eq.forest.1} is bounded from below by $b_{B_{i-1}}(\gamma) \cdot \exp(-c'\cdot L^d)$, with $c'$ another positive constant and for any $\gamma$ satisfying 
$|\gamma|\le L^d R_{i-1}^2$. Concerning the other terms appearing in the sum in~\eqref{eq.forest.1}, by considering the event that the $\ell$-th vertex of the spine has no children in the past, and at least one in the future, we obtain that for some constant $c>0$ whose value may change from line to line, 
\begin{align*}
 \mathbb P (\til{\cT}^{\gamma(\ell)}_+ \text{ hits } K_{i-2} \text{ before } \partial B_i,\, \til{\cT}^{\gamma(\ell)}_-\cap (B_{i-1}\cup \partial B_i)=\emptyset)& \ge c \cdot \inf_{x\sim \gamma(\ell)} \mathbb P \big(\cT^x_c  \text{ hits } K_{i-2} \text{ before } \partial B_i\big)\\
& \ge \frac{c}{L^{d-2}} \cdot \frac{\cpc{K_{i-2}}}{R_{i-1}^{d-2}},  
\end{align*}
using also Lemma~\ref{lem.pre.1} for the last inequality. Altogether this gives, using in addition Lemma~\ref{lem.pre.2}, 
\begin{align*}
& \sum_{w\in \partial B_{i-1}} \mathbb P(\cF_+^w[0,\sigma] \text{ hits } K_{i-2} \text{ before } \partial B_i, \, \cF^w_-[0,\sigma]  \cap (B_{i-1}\cup \partial B_i)=\emptyset\big)  \\
& \ge c\cdot \frac{\cpc{K_{i-2}}}{R_{i-1}^{d-4}}\cdot \exp(-C\cdot L^d) \cdot  \sum_{w\in \partial B_{i-1}} \sum_{\substack{\gamma : w\to \Sigma \\  R_{i-1}^2 \le |\gamma|\le L^dR_{i-1}^2}} s(\gamma) \cdot b_{B_{i-1}}(\gamma) \\
& \ge c \cdot \cpc{K_{i-2}} \cdot \exp(-C\cdot L^d).
\end{align*}
Plugging this together with~\eqref{inf.sigma} into~\eqref{lower.prop.proof}, and then in~\eqref{shell-3} and~\eqref{shell-2} concludes the proof of the proposition. 
\end{proof}


\section{Proofs of miscellaneous  corollaries} \label{sec.cor}

\begin{proof}[\bf Proof of Corollary~\ref{cor-subset}]
We follow broadly the same proof as in \cite{AS23a}, but use    
some simplified arguments. We therefore
omit similar details and focus on the differences. We recall that $B(\cU,r)=\cup_{x\in \cU} B(x,r)$.
The idea is to show that with positive probability there is a set $\cU$ such that 
\[
\cpc{B(\cU,r)}\asymp r^{d-4}\cdot |\cU|\asymp \cpc{B(\cC,r)}.
\]
In \cite{AS23a}, the random subset $\cU$ is constructed by keeping the points $x$ in~$\cC$ such that an independent random walk started from $x$ never returns to $B(\cC,r)$ after escaping the ball $B(x,2r)$. 
In our setting it is in fact slightly simpler to choose a family of independent
Bernoulli variables $\{Y_x,\ x\in \cC\}$ with respective parameter
\[
\bE[Y_x]:=\frac{c}{r^{d-4}} \sum_{y\in \partial B(x,r)} \bP\big(\cT^y_-\cap B(\cC,r)=\emptyset\big),
\]
with $c>0$ chosen so that $\sup_x \bE[Y_x]\le 1$. This is possible, since
\begin{align*}
& \bP\big(\cT^y_-\cap B(\cC,r)=\emptyset\big)\le  \bP\big(\cT^y_-\cap B(x,r)=\emptyset\big) \quad \text{ and } \\
&\sum_{y\in \partial B(x,r)}  \bP\big(\cT^y_-\cap B(x,r)=\emptyset\big)=\cpc{B(x,r)}\asymp r^{d-4}.
\end{align*}
Now, define $\cU=\{x\in \cC:\ Y_x=1\}$. Then, 
\begin{equation}\label{proof-subset1}
\begin{split}
\bE[|B(\cU,r)|]=&|B(0,r)|\cdot \sum_{x\in \cC} 
\frac{c}{r^{d-4}} \sum_{y\in \partial B(x,r)} \bP\big(\cT^y_-\cap B(\cC,r)=\emptyset\big)\\
=& |B(0,r)|\cdot \frac{c}{r^{d-4}} \cpc{B(\cC,r)}\asymp r^4\cdot \cpc{B(\cC,r)}.
\end{split}
\end{equation}
As a sum of Bernoulli, we also obtain Var$(|B(\cU,r)|)\le |B(0,r)|\cdot \bE[|B(\cU,r)|]$, so that
\begin{equation}\label{proof-subset2}
\bP\big( |B(\cU,r)| \asymp r^4\cdot \cpc{B(\cC,r)}\big) \ge \frac{3}{4}.
\end{equation}
Now, we need to deal with the branching capacity of $B(\cU,r)$.  Note that from the lower bound
of the variational characterisation, there is a constant $C$ (independent of $\cC$),
\begin{equation}\label{proof-subset3}
\sum_{x,x'\in \cU}\sum_{y\in \partial B(x,r)}\sum_{y'\in \partial B(x',r)} G(y-y')\ge 
\frac{C\big(|\partial B(0,r)|\cdot |\cU|\big)^2 }{\cpc{B(\cU,r)}}.
\end{equation}
Following the arguments of \cite{AS23a}, we only need an upper bound of
the left hand side of \reff{proof-subset3} of order $r^{d+2}\cdot |\cU|$. 
To obtain an upper bound for the left hand side of \reff{proof-subset3}, we consider expectation,
and treat separately the cases $x=x'$ and $x\not= x'$. 
Assume $x=x'$, then an easy computation using $G(z)\le C \|z\|^{4-d}$ yields (for some
constant $c$ whose value may change from line to line).
\begin{equation}\label{proof-subset4}
\E{ \sum_{x\in \cU} \sum_{y\in \partial B(x,r)}\sum_{y'\in \partial B(x,r)} G(y-y')}
\le c\cdot r^3\cdot r^{d-1}\cdot \E{|\cU|}.
\end{equation}
In the case $x\not= x'$,
\begin{equation}\label{proof-subset5}
\begin{split}
&\E{ \sum_{x\in \cU}\sum_{x'\not =x\in \cU}
\sum_{y\in \partial B(x,r)}\sum_{y'\in \partial B(x',r)}\!\!\! G(y-y')}\le 
c |\partial B(0,r)|^2\sum _{x\not =x'\in \cC}\bP(Y_x=1) G(x-x') \bP(Y_{x'}=1)\\
&\le c \frac{(r^{d-2})^2}{r^{d-4}}  \cdot \bE[|\cU|] \cdot \sup_{x\in \cC}  \sum_{x'\not =x\in \cC} G(x-x')
\sum_{y'\in \partial B(x',r)} \bP\big(\cT^y_-\cap B(\cC,r)=\emptyset\big).
\end{split}
\end{equation}
But since $x\not= x'$, we have that $G(x-x')\asymp G(x-y')$ for any $y'\in \partial B_r(x')$.
Thus, 
\begin{equation}\label{proof-subset6}
\begin{split}
&\E{ \sum_{x\not =x'\in \cU}
\sum_{y\in \partial B(x,r)}\sum_{y'\in \partial B(x',r)} G(y-y')} \\ &\le
c\frac{r^{2d-2}}{r^{d-4}}\cdot \E{|\cU|} \cdot  \sup_{x\in \cC}  \sum_{x'\not =x\in \cC} 
\sum_{y'\in \partial B(x',r)} G(x-y')\bP\big(\cT^y_-\cap B(\cC,r)=\emptyset\big)\\
&\lesssim  r^{d+2}\cdot \E{|\cU|}\cdot   \sum_{z\in \partial B(\cC,r)} G(z-x) e_{B(\cC,r)}(z)\lesssim r^{d+2} \cdot \bE[|\cU|].
\end{split}
\end{equation}
This now proves the desired upper bound on the left hand side of~\eqref{proof-subset3} and finishes the proof of the corollary.
\end{proof}

\begin{proof}[\bf Proof of Corollary~\ref{cor.locset}]
Let $\Lambda$ be a finite and nonempty subset of $\mathbb Z^d$, and consider the function $\varphi(x) =\tfrac{ \1(x\in \Lambda)}{\sup_{y\in \Z^d}G(y,\Lambda)}$. 
It is then immediate that~$\|G\varphi\|_\infty\leq 1$, and thus the corollary follows from Chebyshev's exponential inequality together with Theorem~\ref{theo.expmoment}. 
\end{proof}

\begin{proof}[\bf Proof of Corollary~\ref{thm:variationalchar}]
Define the function $\overline G(x,y)$, by 
$$\overline G(x,y) = \sum_{z\in \mathbb Z^d} g(x,z)g(z,y),$$
which is symmetric and positive definite. Note also that by~\eqref{convol.G} it is of the same order as $G$. Then define 
the scalar product on the set of functions supported on $K$, by  
\begin{equation*}
\langle f,g\rangle=\sum_{x,y\in K} \overline G(x,y) f(x)g(y).
\end{equation*}
As already seen, the upper bound 
$$\inf \big\{\langle \nu, \nu \rangle : \nu \text{ probability measure on }K\big\} \lesssim \frac 1{\cpc{K}}, $$
follows from~\eqref{Upper.theo} by choosing for $\nu$ the measure $\widehat e_K=e_K / \cpc{K}$. 
For the lower bound, note that by~\eqref{Lower.theo} one has for any measure $\nu$ supported on $K$, 
$$\langle \nu,\widehat e_K\rangle \gtrsim \frac{1}{\cpc{K}}.$$ 
On the other hand by Cauchy-Schwarz inequality one also has 
$$\langle \nu,\widehat e_K\rangle^2 \leq \langle \nu,\nu \rangle\cdot \langle \widehat e_K,\widehat e_K\rangle \lesssim \frac{\langle \nu,\nu \rangle}{\cpc{K}}, $$
using again~\eqref{Upper.theo} for the last inequality. Combining the last two displays gives as wanted
$$\inf \big\{\langle \nu, \nu \rangle : \nu \text{ probability measure on }K\big\} \gtrsim \frac 1{\cpc{K}}. $$
\end{proof}

\begin{proof}[\bf Proof of Corollary~\ref{cor:variation}]
We only prove~\eqref{caract.2}, the other characterisation~\eqref{caract.3} is entirely similar and left to the reader. 
The lower bound is obtained by taking $\varphi = \frac{e_K}{\|Ge_K\|_\infty} $. For the upper bound, note that for any function $\varphi$ which is nonnegative on $K$ and 
 satisfies  
$\max_{x\in K} G\varphi (x) \le 1$, 
one has on one hand by~\eqref{Lower.theo}  
$$\langle \varphi, e_K\rangle \gtrsim \sum_{x\in K} \varphi(x),$$
and on the other hand using Cauchy-Schwarz's inequality, 
$$\langle \varphi, e_K\rangle^2 \le \langle \varphi, \varphi\rangle\cdot \langle e_K, e_K\rangle \lesssim \big(\sum_{x\in K} \varphi(x)\big) \cdot \cpc{K},$$ 
which gives the desired upper bound after simplifying. 
\end{proof}

\begin{proof}[\bf Proof of Corollary~\ref{cor.lowerbcap}]
On one hand \eqref{Upper.theo} shows that for any finite set $K$, 
\[
\sum_{x\in K}\sum_{y\in K} G(x,y) e_K(y)\gtrsim |K|,
\]
and on the other hand, by summing first over $x$ and using~\eqref{Green.asymp} we get  
\[
\sum_{x,y\in K} G(x,y) e_K(y)\lesssim  |K|^{4/d}\cdot \cpc{K}.
\]
The corollary follows. 
\end{proof}

\end{document}